\documentclass{amsart}
\usepackage[utf8]{inputenc}

\usepackage{amsmath}
\usepackage{amsthm}

\usepackage[colorlinks=true, allcolors=blue]{hyperref}
\usepackage{cleveref}

\usepackage{MnSymbol}
\usepackage{wasysym}
\usepackage{graphicx}
\usepackage{tikz}
\usepackage{setspace}
\usepackage{float}

\author{Alfredo Roque Freire}
\author{Manuel A. Martins}
\title{Regular non-normal modal classicalities} 

\newtheorem{definition}{Definition}[section]
\newtheorem*{definition*}{Definition}
\newtheorem{theorem}{Theorem}[section]
\newtheorem{lemma}{Lemma}[section]
\newtheorem{corollary}{Corollary}[section]

\newtheorem{notation}{Notation}[section]

\newcommand{\so}{\rightarrow}

\newcommand{\miff}{\Longleftrightarrow}
\renewcommand{\iff}{\longleftrightarrow}

\newcommand{\Lg}{\mathcal{L}}

\newcommand{\pair}[1]{\langle #1 \rangle}

\newcommand{\mlms}{many-logics modal structure }

\newcommand{\cl}{\mathcal{L}}
\newcommand{\bBox}{\blacksquare}

\newcommand{\mBox}{\boxminus}

\usepackage[backend=bibtex,style=alphabetic,maxbibnames=15,maxcitenames=6,dateabbrev=false]{biblatex}
\newcommand{\B}{$\mathcal{B}_4^\circ$}
\newcommand{\For}{\mathrm{Form}}
\newcommand{\VAR}{\mathrm{Var}}

\newtheorem{remark}{Remark}

\renewcommand\labelenumi{(\roman{enumi})}
\renewcommand\theenumi\labelenumi

\begin{document}

\begin{abstract}
We present a novel investigation into the consistency operator ($\circ$), traditionally associated with paraconsistent logics, as a means of capturing non-normal modal classicalities within the Kripke framework. By semantically reinterpreting $\circ$ as an operator that distinguishes top and bottom values from other values in the algebra, we extend its applicability beyond paraconsistency into classical and modal logics. We introduce the logic $\mathcal{B}_4^\circ$, a four-valued Boolean logic augmented with the consistency operator, and provide a sound and complete axiomatization. Building on this foundation, we extend the semantics to the modal domain using the many-logics modal logic (MLML) framework. Specifically, we construct Kripke frames based on an eight-valued Boolean algebra that contains three distinct four-valued subalgebras, each representing a different world type. Our analysis reveals that the resulting modal logic exhibits a normal local consequence relation alongside a non-normal global consequence relation. Consequently, the characterization of frame properties --such as transitivity, reflexivity, and Euclideanness --deviates from modal logic $K$, requiring novel semantic tools. We further identify new modal formulas in an extended language that capture previously unavailable kinds of accessibility, leading to frame characterizations unattainable in traditional modal frameworks.
\end{abstract}

\maketitle

\vspace{0.5cm}

\textit{This article is dedicated to our dear friend and professor, Marcelo Coniglio, whose guidance has been and continues to be a source of great learning for us.}

\section{Introduction}

This article deals with what can be expressed with the $\circ$ operator (often called the consistency operator) in a novel context, namely, to express non-normal modal consequence relations within Kripke frames. To achieve this goal, we develop three ideas devised in the following three subsections.

\subsection{A semantical abstraction of the consistency operator}

The consistency operator was originally used in the formalization of $C_n$ paraconsistent systems by Newton da Costa \cite{da_costa_theory_1974, costa_sistemas_1993}. The operator is used in $C_n$ to recover the \textit{proof by contradiction} for particular formulas within $C_n$'s consequence relation. This is achieved with the following axiom \cite[p. 499]{da_costa_theory_1974}:
\begin{center}
    Axiom 12 - $\overbrace{\circ \ldots \circ}^n B \so ((A \so B) \so ((A \so \lnot B) \so \lnot A))$
\end{center}
In $C_1$, obtaining the inconsistency $B$ and $\lnot B$ from $A$ allows us to conclude $\lnot A$ only when $\circ B$ occurs. It is later easy to observe that $C_n$ trivializes in the presence of 
$$\overbrace{\circ \ldots \circ}^n X \land X \land \lnot X$$
A way of recovering trivialization is important. A paraconsistent system where trivialization cannot be recovered legs behind in capturing some important mathematical and scientific reasoning. 

Da Costa only took $\circ X$ to be the abbreviation of $\lnot (X \land \lnot X)$, so $\circ$ did not have an independent interpretation as observed in later research. 
With the works of Carnielli, Coniglio and Marcos \cite{marcos_logics_2005, carnielli_logics_2007}, we observe a more abstract take on the $\circ$ operator. Instead of approaching $\circ \varphi$ simply as the abbreviation of $\lnot (\varphi \land \lnot \varphi)$, they consider it more abstractly as the operator that recovers the principle of explosion for the particular formula over some basic formulation of a paraconsistent system. 
This opens up the study of the $\circ$ operator within a comprehensive class of paraconsistent systems.

The present paper sketches yet another kind of abstraction for the $\circ$ operator. Instead of viewing it through its role in a formal system, we investigate the operator semantically. We will view the operator as responsible for filling expressive gaps in a many-valued semantics.
Consider the example of the \textit{logic of paradox} (LP) introduced by Priest in \cite{priest_logic_1979}. When extended with $\circ$ operator, this logic results in the LFI1 system\footnote{This requires us to remove the implication from the logical language; but in both LP and LFI1 the implication is definable with the other logical operators.}  (see. \cite[p. 158-167]{carnielli_paraconsistent_2016}). The three-valued lattice semantics with values $\{1,c,0\}$ for LP is such that $0 < c < 1$ and $-1 = 0$, $-0 = 1$ and $-c = c$. Now, what are the subsets of $\{1,c,0\}$ that are characterizable by a LP formula? Indeed, we have
\begin{enumerate}
    \item $\{1,c,0\} = \{v(p) \mid v \text{ is a LP valuation}\}$.
    \item $\{1,c\} = \{v(p \lor \lnot p) \mid v \text{ is a LP valuation}\}$.
    \item $\{c, 0\} = \{v(p \land \lnot p) \mid v \text{ is a LP valuation}\}$.
\end{enumerate}
Other subsets of semantical values are not characterizable in the same way. However, once we add the operator $\circ$ such that $\circ (1) = 1$, $\circ (0) = 1$ and $\circ (c) = 0$, we are now capable of characterizing all non-empty subsets of $\{1, c, 0\}$:
\begin{enumerate}
    \item $\{1\} = \{v(\lnot \circ p \lor p \lor \lnot p) \mid v \text{ is a LP valuation}\}$.
    \item $\{0\} = \{v(\circ p \land p \land \lnot p) \mid v \text{ is a LP valuation}\}$.
    \item $\{c\} = \{v(p) \mid v \text{ is a LP valuation}\}$.
    \item $\{1,0\} = \{v(\circ p) \mid v \text{ is a LP valuation}\}$.
\end{enumerate}
Of course, we may have added another operator that fills the same role of bridging the semantic gap. There is, nonetheless, something special about having this formulation. An operator that identifies top and bottom values is, firstly, something that can be more generally studied for other semantics (e.g. all finite lattice semantics have top and bottom values) and, secondly, this kind of operator has been largely studied in the recent paraconsistent literature. By considering the semantic interpretation of $\circ$, we open up the possibility of adding this operator to logics other than the paraconsistent ones. A logic that allows for trivialization may also find new expressible ideas by considering the new operator distinguishing the top and bottom values from the other values in their semantics.

\subsection{Extending the use of \texorpdfstring{$\circ$}{consistency operator} to classical logic} 

In this article, we shall observe that even within classical logic and classical modal logic, there is a meaningful role for the $\circ$ operator. As well documented, classical logic is algebraized by Boolean algebras, where an ultrafilter determines the valid values \cite{blok_algebraizable_1989}. This implies that every classical tautology yields values within the ultrafilter in any Boolean interpretation of the propositional variables. Similarly, for any finite consequence relation, if the premises take values in the ultrafilter, then so does the consequence. This general result may suggest that Boolean semantics with more than two values offer no additional semantic richness.

However, by considering a Boolean algebra with more than two values, we create opportunities to introduce additional operators that enhance expressivity within the Boolean lattice. In this context, we investigate four-valued Boolean algebras augmented with the $\circ$ operator. Analogously to its interpretation in LFI1, the $\circ$ operator is assigned the top value for the top and bottom values, while it takes the bottom value for all others.

To avoid unnecessary complexity, we focus on the four-valued Boolean algebra as a case study. In formalizing this semantics (see \cref{4-boolean}), we shall observe that the $\circ$ operator can be derived even in the absence of premises explicitly containing $\circ$. For instance, $\circ \varphi$ follows for every logical theorem $\varphi$, and every formula of the form $\circ \circ \psi$ is itself a theorem. This phenomenon naturally arises within the classical framework, unlike its traditional role in paraconsistent logics.

\subsection{Non normal global modal consequence relation}

In the second part of this paper, we apply the strategy of many logic modal logic systems \cite{freire_modality_2024} to derive a \textbf{non-normal modal global consequence relation} where every world is classical. A simple example illustrates a key difference: in classical normal modal logic, $\Box p$ is a global consequence of $p$, whereas in our proposed system, this will no longer hold.

The study of non-normal modal systems began with Dana Scott in a series of unpublished works (see \cite{lemmon_algebraic_1966, lewis_counterfactuals_1973} for Scott's attribution in key theorems). Subsequent investigations were carried out by Kripke \cite{kripke_semantical_1963} and many others (see \cite{berto_impossible_2019, priest_introduction_2008} for further developments). 
The failure of normality in our system comes from a broader perspective of necessity. Scott’s original approach introduces the \textbf{neighborhood semantics}, where modal necessity is determined by neighborhood structures rather than Kripkean accessibility relations\footnote{The neighborhood semantics framework was later extensively developed and defended by Lewis; see \cite{lewis_general_1970, lewis_counterfactuals_1973}.}. An alternative approach, which we adopt here, retains the standard accessibility relation but incorporates \textbf{non-normal worlds}—worlds where the truth valuation does not necessarily follow standard recursive clauses.

Our formulation introduces non-normality \textbf{even when all worlds are classical}. This is achieved by allowing the classical worlds to disagree on how to interpret evaluations made by other classical worlds. To formalize this, we employ an eight-valued Boolean algebra, which decomposes into three distinct four-valued Boolean algebras. Each of these four-valued structures represents a different type of world: $A$, $B$, and $C$. In this setting, a world of type $A$ considers $p$ to be possible in another world of type $A$ if and only if $p$ is valid. However, for worlds of types $B$ and $C$, $p$ is deemed possible only when it is valid with the top value.

With this approach, we will show that the \textbf{local consequence relation} remains identical to that of classical Kripke semantics. However, the \textbf{global consequence relation} deviates from normality, leading to a different strategy and results for frame characterization (e.g. transitivity, reflexivity, Euclidean, etc.). In the final section, we explore these implications in detail and demonstrate how our system enables novel types of frame characterizations beyond those available in traditional modal logics.

\subsection{Outline of the paper}

In Section \ref{sec_4-valued}, we introduce the algebra $B_4^\circ$, which is defined by a $4$-valued Boolean algebra with operation $\circledcirc$ and an ultrafilter over it. We provide a sound and complete axiomatization for this semantics. Then, in Section \ref{MLML}, we briefly explore the framework of many-logics modal systems introduced in \cite{freire_modality_2024}. Section \ref{sec:four-eight-logics} constitutes the core of the paper, where we examine the many-logics modal logic resulting from considering the $8$-valued Boolean algebra $B_8^\circ$ as our base lattice. The worlds in this semantics are associated with one of the three $4$-valued sublattices of $B_8^\circ$ and the necessity operator is defined according to the order relation in $B_8^\circ$. In Section~\ref{sec_prop}, we examine the characterization of frame properties using modal formulas. Traditional accessibility relations, such as reflexivity and transitivity, remain characterizable by the standard axioms $T$ and $4$ of traditional modal logic. However, properties such as Euclideanness require the use of the $\circ$ operator. We also investigate new accessibility relations that involve changes to the base lattice. For instance, we consider the characterization of `accessing a world in a different lattice', as well as a version of `transitivity restricted to worlds sharing the same lattice'. The paper concludes with a summary of our results and a discussion of potential directions for future research.

\section{The \texorpdfstring{$4^\circ$}{4-ball}-valued Boolean logic}\label{sec_4-valued}

Let $\For(\mathcal{L}^\circ)$ be the set of propositional formulas built from a set of variables $\VAR$ and the set of symbols $\mathcal{L}^\circ$ with unary connectives $\neg$ and $\circ$ and binary connectives $\wedge$ and $\vee$.

Throughout this paper, we adopt the following notation for Boolean algebras: $1$ denotes the top value, $0$ the bottom value, $+$ the join operation, $\cdot$ the meet operation and $-$ the complement.

Given a Boolean algebra $B$, we expand it by introducing a new operator $\circledcirc$, defined as follows:
$$\circledcirc x = \begin{cases}
	1, \text{ if } x \in \{1, 0\}\\
	0, \text{ otherwise}
\end{cases}$$
The algebra obtained by adding to the $n$-valued Boolean algebra the operator $\circledcirc$ will be called the $n^\circ$-valued Boolean algebra. We use $B^\circ_n$ to refer to this Boolean algebra.

\begin{definition}
    A valuation in a $n^\circ$-valued Boolean algebra $B^\circ_n$ is a function $v: \For(\cl^\circ) \longrightarrow B^\circ_n$ such that 
		\begin{enumerate}
			\item $v(\lnot \varphi) = - v(\varphi)$
			\item $v(\circ \varphi) = \circledcirc (v(\varphi))$
			\item $v(\varphi \lor \psi) = v(\varphi) + v(\psi)$
			\item $v(\varphi \land \psi) = v(\varphi) . v(\psi)$.
		\end{enumerate}
    Furthermore, we define the connectives $\varphi \so \psi$ and $\varphi \iff \psi$ as abbreviations of $(\lnot \varphi \lor \psi)$ and $((\varphi \so \psi) \land (\psi \so \varphi))$, respectively.
\end{definition}

We should also consider a binary evaluation of formulas in a satisfaction relation. For that, we make use of the ultrafilter $U$ associated with the Boolean algebra $B^\circ_n$.

\begin{definition}
	Let $U$ be an ultrafilter over $B^\circ_n$.
	Given a valuation $v$, we say that $v$ satisfies a formula $\varphi$ when $v(\varphi) \in U$. In this case, we write $v \vDash_{B^\circ_n}^U \varphi$.
	Moreover, we say that $\varphi$ is valid (in symbols, $\vDash_{B^\circ_n}^U \varphi$) when for every valuation $v$ $v \vDash_{B^\circ_n}^U \varphi$.
\end{definition}

When $B$ and $U$ are clear from the context, we just write $\vDash$ for $\vDash_{B^\circ_n}^U$. 

\subsection{The logic of \texorpdfstring{$4^\circ$}{4-ball}-Boolean Algebra}\label{4-boolean}

Next, we take a closer look at the case where we consider $B^\circ_4$ together with an ultrafilter $U$ over $B^\circ_4$. We will show that the corresponding logic, denoted by $\mathcal{B}$, can be axiomatized by the following set of axioms and inference rules:

\begin{enumerate}
    \item\label{lc} (CL) Classical logic axioms and rules for the extended language. (Classical logic)
    \item\label{double-ball} (DB) $\vdash \circ \circ \varphi$. (Double ball axiom)
    \item\label{ball-rule} (BR) $\circ \varphi \dashv\vdash \circ \lnot \varphi$. (Ball rule)
    \item\label{ball-factorization} (BF) $\circ \varphi, \circ \psi \vdash \circ (\varphi \land \psi)$; $\circ \varphi, \circ \psi \vdash \circ (\varphi \lor \psi)$. (Ball factorization rule)
    \item\label{affirming-ball-rule} (AwB) $\varphi, \circ \varphi \vdash \circ (\varphi \lor \psi)$ (Affirming with ball rule)
    \item\label{negating-ball-rule} (NwB) $\lnot \varphi, \circ \varphi \vdash \circ (\varphi \land \psi)$  (Negating with ball rule)
    \item\label{not-ball-rule} (NB) $\lnot \circ \varphi, \circ \psi \vdash \lnot \circ(\varphi \land \psi) \lor \lnot \circ (\varphi \lor \psi)$. (Not ball rule)
    \item\label{two-not-ball-rule} (TNB1) $\lnot \circ \varphi, \lnot \circ \psi, \varphi \land \psi \vdash \lnot \circ(\varphi \land \psi)$. (Two not ball rule 1)
    \item (TNB2) $\lnot \circ \varphi, \lnot \circ \psi, \lnot (\varphi \lor \psi) \vdash \lnot \circ(\varphi \land \psi)$. (Two not ball rule 2)
    \item\label{ball-conjunction} (BC) $\circ (\varphi \land \psi), \varphi \land \psi \vdash \circ \varphi \land \circ \psi$. (Ball and conjunction rule)
    \item\label{opposite-value} (OV) $\circ \varphi \iff \circ \psi, \varphi \iff \lnot \psi \vdash \circ (\varphi \land \psi)$. (Opposite value rule)
    \item\label{intro-ball} (IB) if $\vdash \varphi$, then $\vdash \circ \varphi$. (Introduction of ball)
\end{enumerate}

The operator $\circ \varphi$ can be interpreted as saying that the truth value of $\varphi$ is certain or fully known. This aligns with the epistemic interpretation of paraconsistency proposed by Carnielli and Rodrigues \cite{carnielli_epistemic_2019, rodrigues_paraconsistency_2023}, where logical inconsistency is handled in terms of evidence and knowledge. In a modal context, applying $\circ$ to a formula serves in the current context as a guarantee, across alternative possible worlds, that the formula should be regarded as definitively true or definitively false. Since we are working within a classical framework, the resulting systems differ in how they handle conflicting evidence: while Carnielli and Rodrigues' approach allows for the coexistence of contradictory information, the present system does not permit such conflicts.

\begin{theorem}[Sound and complete]\label{sound-complete-10-classic}
    For every set of formulas $\Gamma \cup \{\varphi\}\subseteq \For(\cl^\circ)$,
    $$\Gamma \vDash \varphi \text{ if, and only if, } \Gamma \vdash \varphi$$
\end{theorem}

Before proving this result, we should introduce the notion of a maximal consistent set and present some properties of it.

\begin{definition}
    Let $\Gamma \subseteq \For(\cl^\circ)$. We say that $\Gamma$ is a consistent set in the \B system if there is a formula $\varphi$ such that $\Gamma \not \vdash \varphi$.
\end{definition}

\begin{definition}
    Let $\Sigma$ be a consistent set in the \B-system. We say that $\Gamma$ is a maximally consistent \B-extension of $\Sigma$ if
    \begin{enumerate}
        \item $\Sigma \subseteq \Gamma$.
        \item $\Gamma$ is \B-consistent.
        \item $\Gamma$ is closed under the \B-rules.
        \item $\Gamma$ is maximal, i.e. $\Gamma \subseteq \Lambda$ and $\Lambda$ consistent implies $\Gamma = \Lambda$.
    \end{enumerate}
\end{definition}

\begin{lemma}
    Let $\Gamma$ be a maximally consistent \B-extension, then $\varphi \in \Gamma$ or $\lnot \varphi \in \Gamma$ for every formula $\varphi$ in $\For(\Lg^\circ)$.
\end{lemma}

\begin{proof}
    Suppose that $\varphi \notin \Gamma$ and $\lnot \varphi \notin \Gamma$. Then $\Gamma\cup\{\varphi\}$ is consistent (by classical rules), which contradicts the fact that $\Gamma$ is maximal.
\end{proof}

\begin{definition}
    Let $A = \pair{1, 0, a, -a}$ be a four valued Boolean algebra such that $a$ is in the ultrafilter $U$ and let $\Gamma$ be a maximally consistent \B-extension, we define the $A$-canonical valuation $v:Var\to A$ with respect to $\Gamma$ by
    \begin{enumerate}
        \item $v(X) = 1$ when $X \in \Gamma$ and $\circ X \in \Gamma$.
        \item $v(X) = 0$ when $X \notin \Gamma$ and $\circ X \in \Gamma$.
        \item $v(X) = a$ when $X \in \Gamma$ and $\circ X \notin \Gamma$.
        \item $v(X) = -a$ when $X \notin \Gamma$ and $\circ X \notin \Gamma$.
    \end{enumerate}
\end{definition}

\begin{lemma}\label{conservation-values-10-classical}
    Let $A = \pair{1, 0, a, -a}$ be a four valued Boolean algebra such that $a \in U$, let $\Gamma$ be a maximally consistent \B-extension and $v$ the $A$-canonical valuation for $\Gamma$, then, for every formula $\varphi$,
    \begin{enumerate}
        \item $v(\varphi) = 1$ if, and only if, $\varphi \in \Gamma$ and $\circ \varphi \in \Gamma$.
        \item $v(\varphi) = 0$ if, and only if, $\varphi \notin \Gamma$ and $\circ \varphi \in \Gamma$.
        \item $v(\varphi) = a$ if, and only if, $\varphi \in \Gamma$ and $\circ \varphi \notin \Gamma$.
        \item $v(\varphi) = -a$ if, and only if, $\varphi \notin \Gamma$ and $\circ \varphi \notin \Gamma$.
    \end{enumerate}
\end{lemma}

\begin{proof}
    We prove this by induction.

    For variables, it holds by definition.
    
    Suppose $\varphi$ is $\lnot \alpha$.
    \textbf{(i)} If $v(\varphi) = 1$, then $v(\alpha) = 0$. By IH, $\alpha \notin \Gamma$ and $\circ \alpha \in \Gamma$. Since $\Gamma$ is maximal, $\lnot \alpha \in \Gamma$; from \ref{ball-rule}-BR, $\circ \lnot \alpha \in \Gamma$. If, on the other hand, $\lnot \alpha \in \Gamma$ and $\circ \lnot \alpha \in \Gamma$, then $\circ \lnot \alpha \in \Gamma$ from \ref{ball-rule}-BR and $\alpha \notin \Gamma$ since $\Gamma$ is consistent. Consequently, $v(\alpha) = 0$ from IH and therefore $v(\lnot \alpha) = 1$.
    \textbf{(ii)} If $v(\varphi) = a$, then $v(\alpha) = -a$. By IH, $\alpha \notin \Gamma$ and $\circ \alpha \notin \Gamma$. Thus, $\neg \alpha \in \Gamma$. From \ref{ball-rule}-BR, $\circ \lnot \alpha \in \Gamma$ implies $\circ \alpha \in \Gamma$; consequently, from $\circ \alpha \notin \Gamma$, we obtain $\circ \lnot \alpha \notin \Gamma$. Now, if $\lnot \alpha \in \Gamma$ and $\circ \lnot \alpha \notin \Gamma$, then $\alpha \notin \Gamma$ since $\Gamma$ is consistent and $\circ \alpha \notin \Gamma$ from \ref{ball-rule}-BR. From IH, we obtain $v(\alpha) = -a$ and thus $v(\lnot \alpha) = a$.
    For the cases \textbf{(iii)} $v(\varphi) = 0$ and \textbf{(iv)} $v(\varphi) = -a$, the proof is similar.

    Suppose $\varphi$ is $\alpha \land \beta$. 
    \textbf{(i)} If $v(\alpha \land \beta) = 1$, then $v(\alpha) = 1$ and $v(\beta) = 1$. From IH, we obtain $\alpha \in \Gamma$ and $\beta \in \Gamma$; also from IH, we obtain $\circ \alpha \in \Gamma$ and $\circ \beta \in \Gamma$. From \ref{lc}-CL, we have $\alpha \land \beta \in \Gamma$ and, from \ref{ball-factorization}-BF, we have $\circ (\alpha \land \beta) \in \Gamma$. If $\alpha \land \beta \in \Gamma$ and $\circ (\alpha \land \beta) \in \Gamma$, then $\circ \alpha \land \circ \beta \in \Gamma$ from \ref{ball-conjunction}-BC. Additionally, we obtain $\alpha \in \Gamma$, $\beta \in \Gamma$, $\circ \alpha \in \Gamma$ and $\circ \beta \in \Gamma$ from \ref{lc}-CL. Hence $v(\alpha) = 1$ and $v(\beta) = 1$ from IH. 
    \textbf{(ii)} If $v(\alpha \land \beta) = 0$, then (ii-a) $v(\alpha) = -v(\beta) = a$, (ii-b) $v(\beta) = - v(\alpha) = a$, (ii-c) $v(\alpha) = 0$ or (ii-d) $v(\beta) = 0$. In case we have (ii-a), IH give us $\alpha \in \Gamma$ and $\beta, \circ \alpha, \circ \beta \notin \Gamma$. Then we obtain $\lnot \beta \in \Gamma$ and 
    $\circ \alpha \iff \circ \beta \in \Gamma$. 
    Consequently, that $\alpha \iff \lnot \beta \in \Gamma$. From \ref{opposite-value}-OV, $\circ (\alpha \land \beta) \in \Gamma$. From $\lnot \beta \in \Gamma$, we obtain $\alpha \land \beta \notin \Gamma$. In case we have (ii-b), the same argument applies. In case (ii-c), the IH give us $\alpha \notin \Gamma$ and $\circ \alpha \in \Gamma$. This already give us $\alpha \land \beta \notin \Gamma$ from \ref{lc}-LC. Then, from \ref{negating-ball-rule}-NwB, we have $\circ (\alpha \land \beta) \in \Gamma$. Finally, we obtain (ii-d) using the same strategy.  Now suppose $\alpha \land \beta \notin \Gamma$ and $\circ (\alpha \land \beta) \in \Gamma$. Towards contradiction, assume $v(\alpha \land \beta) \neq 0$. If $v(\alpha \land \beta) = 1$, then we obtain (from IH) $\alpha, \beta \in \Gamma$ and, consequently, the absurd $\alpha \land \beta \in \Gamma$. If $v(\alpha \land \beta) = a$, then we again use IH to obtain $\alpha, \beta \in \Gamma$; this leads to the same absurd conclusion. Lastly, consider $v(\alpha \land \beta) = -a$. If both $v(\alpha)$ and $v(\beta)$ are $-a$, we obtain from IH that $\lnot \circ \alpha \in \Gamma$ and $\lnot \circ \beta \in \Gamma$. From \ref{ball-rule}-BR and $\Gamma$'s maximality, we conclude $\lnot \circ \lnot \alpha \in \Gamma$ and $\lnot \circ \lnot \beta \in \Gamma$. Since we also obtain $\lnot (\alpha \lor \beta) \in \Gamma$, then we obtain from \ref{two-not-ball-rule}-TNB the absurd $\lnot \circ (\alpha \land \beta) \in \Gamma$. Now if (without loosing generality) $v(\alpha)=-a$ and $v(\beta) = 1$, then $\lnot \circ \alpha, \circ \beta \in \Gamma$. Hence, from \ref{affirming-ball-rule}-AwB and \ref{not-ball-rule}-NB, we obtain the absurd $\lnot \circ (\alpha \land \beta) \in \Gamma$.
    \textbf{(iii)} If $v(\alpha \land \beta) = a$, then (without loss of generality) $v(\alpha) = a$ and $v(\beta)$ is $1$ or $a$. From IH, $\alpha, \beta \in \Gamma$ and $\lnot \circ \alpha \in \Gamma$. Using \ref{not-ball-rule}-NB,\ref{two-not-ball-rule}-TNB and \ref{affirming-ball-rule}-AwB we obtain that $\alpha, \beta, \lnot \circ \alpha \vdash \lnot \circ(\alpha \land \beta)$; thence $\alpha \land \beta \in \Gamma$ and $\lnot \circ (\alpha \land \beta) \in \Gamma$. Now, if $\alpha \land \beta \in \Gamma$ and $\lnot \circ (\alpha \land \beta) \in \Gamma$, then we cannot have $\circ \alpha, \circ \beta \in \Gamma$ because of factorization (\ref{ball-factorization}-BF). So (without loss of generality) $\lnot \circ \alpha \in \Gamma$ from maximality of $\Gamma$. IH gives us that $v(\alpha) = a$; also, since $\beta \in \Gamma$, we obtain $v(\beta) = a$ or $1$. Therefore $v(\alpha \land \beta) = a$. 
    \textbf{(iv) }If $v(\alpha \land \beta) = -a$, then (without loss of generality) we may assume $v(\alpha) = - a$. From IH, $\lnot \alpha, \lnot \circ \alpha \in \Gamma$. For $\beta$ we have two options: (iv-a) $v(\beta) = 1$ or (iv-b) $v(\beta) = -a$. In case (iv-a), we obtain from IH that $\circ \beta, \beta \in \Gamma$. Subsequently, combining \ref{not-ball-rule}-NB and \ref{affirming-ball-rule}-AwB, we have $\lnot \circ (\alpha \land \beta)$. In case (iv-b), we obtain $\lnot \beta, \lnot \circ \beta \in \Gamma$ and that $\alpha \land \beta \notin \Gamma$. Then we have $\lnot (\alpha \lor \beta) \in \Gamma$ and, consequently, using \ref{two-not-ball-rule}-TNB, that $\circ (\alpha \land \beta) \notin \Gamma$. On the other hand, if $\alpha \land \beta \notin \Gamma$ and $\circ (\alpha \land \beta) \notin \Gamma$, then at least one of $\alpha$ or $\beta$ is in $\Gamma$. Without loss of generality, assume $\alpha \notin \Gamma$. Suppose that $v(\alpha \land \beta) \neq a$. In case $v(\alpha \land \beta)$ is $1$ or $a$, then we conclude from IH the absurd $\alpha, \beta \in \Gamma$ since possible values for $\alpha$ and $\beta$ are either both $1$ or one of then is $a$ and the other is $a$ or $1$. So lastly we consider $v(\alpha \land \beta) = 0$; from IH and $\alpha \notin \Gamma$, $v(\alpha)$ is $1$ or $a$. In the first case, we obtain $\circ \alpha$, and from \ref{negating-ball-rule}-NwB we then obtain the absurd $\circ(\alpha \land \beta) \in \Gamma$. In the second case, we obtain that $v(\beta) = a$ and, from IH, that $\beta, \lnot \circ \beta \in \Gamma$. From \ref{opposite-value}-OV, we obtain the absurd $\circ (\alpha \land \beta) \in \Gamma$.

    Suppose $\varphi$ is $\circ \alpha$. Before inspecting the cases, note that $v(\circ \alpha)$ can only be $1$ or $0$ from the definition.
    \textbf{(i)} If $v(\circ \alpha) = 1$, then $v(\alpha)$ is $1$ or $0$; in both cases, the IH gives us $\circ \alpha \in \Gamma$. From \ref{double-ball}-DB, $\circ \circ \alpha \in \Gamma$, completing this case. On the other hand, suppose $\circ \alpha \in \Gamma$. If $v(\circ \alpha) = 0$, then $v(\alpha)$ if $a$ or $-a$; in both cases, IH gives us that $\circ \alpha \notin \Gamma$, which is absurd. Therefore, $v(\circ \alpha) = 1$. Now, \textbf{(ii)} if $v(\circ \alpha) = 0$, we obtain $\circ \alpha \notin \Gamma$. But again we have $\circ \circ \alpha \in \Gamma$ from \ref{double-ball}-DB. Conversely, if $\circ \alpha \notin \Gamma$ and $\circ \circ \alpha \Gamma$, then $v(\alpha) = a$ or $v(\alpha) = -a$ from IH. Consequently, $v(\circ \alpha) = 0$ as desired. Note that no other case need to be considered. First, the value of $\circ \alpha$ can only be $1$ or $0$; second, it is never the case that $\circ \circ \alpha \notin \Gamma$ since $\Gamma$ is consistent.
\end{proof}

\begin{proof}[Proof of Theorem \ref{sound-complete-10-classic}]
    The proof of soundness is obtained by simply observing that each rule of the formal system is obtained semantically. We leave this to the interested reader.
    
    We prove completeness. Counterpositively, we suppose $\Gamma \nvdash \varphi$. Then we find that $\Sigma \cup \{\lnot \varphi\}$ is consistent. Let $\Lambda$ be a maximally consistent extension of $\Gamma$ and $v$ the $A$-canonical valuation with respect to $\Lambda$. From Lemma~\ref{conservation-values-10-classical}, $\alpha \in \Lambda$ if and only if $v(\alpha)$ is $1$ or $a$, with $U = \{1, a\}$. Thus $v(\Lambda) \subset U$ and hence $v(\lnot \varphi) \in U$. It follows that $v(\varphi) \notin U$ and thus $v$ exemplifies $\Gamma \nvDash \varphi$.
    
\end{proof}

\section{Many-logics modal logic}\label{MLML}

In \cite{freire_modality_2024}, we address the challenge of combining modal worlds operating under different logical systems. We propose a framework in which modal structures are connected via a common lattice, enabling the evaluation of modal formulas across worlds with distinct logics. The truth value of a modal formula in a world $w$ is determined by relativizing the values of other logics to the lattice of $w$, using the ordering of the base lattice as a reference. In this section, we review the foundational notions underlying this approach.

The language $\Lg^\Box$ is composed of a set $Var$ of denumerably many sentential letters $p_1, p_2, \dots$, the unary
connective  $\neg$ and $\Box$ and the binary connectives $\land$, $\lor$, 
and parentheses. The set of formulas of $\Lg^\Box$ is recursively defined in the usual way and will be denoted by $\For(L^\Box)$.

A {\it  filtered matrix $\Lg$-lattice} is a pair $\pair{{\bf L},{\rm D}}$, where 
\begin{itemize}
\item 
$ {\bf L}=\pair{L,.,+,-}$ such that  $\pair{{L},.,+}$ is a lattice 
$-$ is an unary operation over $L$ 
(in some cases, we should impose conditions about $-$),
\item ${\rm D} \subseteq L$. 
\end{itemize}
We denote by $SubLat(\pair{{L},.,+})$ the set of all complete sublattices of $\pair{{L},.,+}$. Each $\mathbf{L_0}\in SubLat(\pair{{L},.,+})$ induces a submatrix of $\pair{{\bf L},{\rm D}}$, namely $\pair{{\bf L_0},{\rm D}\cap L_0}$ (note that this induced matrix can be degenerate in the sense that ${\rm D}\cap L_0$ can be empty). 

In what follows, we fix a subset $LAT$ of $SubLat(\pair{{L},.,+}$.

\begin{definition}[Many-logics modal  frames] 
		 A \textbf{many-logics modal  frame} $F$ (over $LAT$ and $\Lg^\Box$) is a tuple $\pair{W, R, I}$ such that
	\begin{enumerate} 
 \item $W$ is a non empty set.
 \item $R \subseteq W \times W$ is a relation from worlds to worlds.
\item $I:W\to LAT$ assigns a lattice to each world. We denote the lattice $I(w)$  by ${\bf L_w}$ and the induced matrix by $\pair{{\bf L_w},{\rm D_w}}$. 
\end{enumerate}
    \end{definition}
    
As in the standard case, if we add a valuation to a frame, we obtain a modal structure.

\begin{definition}[Many-logic modal  structures -- \cite{freire_modality_2024}] 
		 A \textbf{many-logic modal  structure} $M$ (over $LAT$ and $\Lg^\Box$) is a tuple $\pair{W, R, I,v}$ such that
	\begin{enumerate} 
 \item $\pair{W, R, I}$ is a frame (called the frame of $M$)
\item   $v: W \times Var \longrightarrow L$ is a valuation function such that $v(w,p)\in I(w)$, for $w\in W$ and $p\in Var$. We write $v_w(p)$ for $v(w,p)$.
	\end{enumerate}
	\end{definition}

To extend a valuation $v_w$  to all formulas of $\Lg^\Box$, we should define a way to interpret values in a complete sublattice $L' \subseteq L$ not necessarily in $L'$.

\begin{definition}[Down- and up- interpretations] \label{def.down.int} 
	Let $L$ be a lattice and $L'$a complete sublattice of $L$. 

    \begin{enumerate}
        \item The {\bf down-interpretation} of the value $x \in L$ in a lattice $L'$ is defined as follows: $x_{down}^{L'} = \bigvee_{L'} \{y \in L' :  y \leqslant x \} $, if $\{y \in L' :  y \leqslant x \} = \emptyset$, then $x^{L'}$ is the least value in $L'$. 
        \item The {\bf up-interpretation} of the value $x \in L$ in a lattice $L'$ is defined as follows: $ x_{up}^{L'} = \bigwedge_{L'} \{y \in L' :  y \leqslant x \} $, if $\{y \in L' :  y \leqslant x \} = \emptyset$, then $x^{L'}$ is the least value in $L'$. 
    \end{enumerate}
\end{definition}

Note that if $a\in L'$, then $ a_{down}^{L'}=a_{up}^{L'}=a$. In what follows, we will adopt the down interpretation and we just write $ a^{L'}$ for $a_{down}^{L'}$ .
Now, for each world $w$ in a many-logic modal structure $M$, 
the valuation $v_w$ is extended to all formulas of $\Lg^\Box$ in the following way. 

\begin{definition} \label{def.valuation.mlms} 
Let $M$ be a \mlms $M = \pair{W,R,{I},v}$.  For each $w\in W$, the valuation $v_w$ induced by $M$ is the function $v_{w}: \For(\mathcal{L}^\Box) \longrightarrow L_{w}$ such that: 
		\begin{enumerate}
			\item $v_{w}(p) = v(w, p)$, where $p$ is a proposition variable,
			\item $v_{w}(\lnot \alpha) = (- v_{w}(\alpha))^{L_{w}}$,
			\item $v_{w}(\alpha \lor \beta) = (v_{w}(\alpha) + v_{w}(\beta))^{L_{w}}$,
			\item $v_{w}(\alpha \land \beta) = (v_{w}(\alpha).v_{w}(\beta))^{L_{w}}$,
			%\item $v_{w}(\alpha \so \beta) = (v_{w}(\alpha)\Rightarrow v_{w(\beta))}^{L_{w}}$,
                \item $v_{{w}}(\Box \alpha)=\bigwedge\limits_{L_{{w}}}\{(v_{u}(\alpha))^{L_{{w}}} \mid {w} R u \land u \in W\}$,
                \item $v_{w}(\Diamond \alpha) = (- v_{w}(\Box \lnot \alpha))^{L_{w}}$.
 	\end{enumerate}

We say that a formula $\varphi$ holds in a model $M$ in a world $w$, i.e. $M, w \vDash \varphi$ (or simply $w \vDash \varphi$ when $M$ is clear from context), if $v_w(\varphi) \in \mathrm{D}_w$. A formula $\varphi$ is valid in $M = \pair{W, R, I, v}$, i.e. $M \vDash \varphi$, if $M, w \vDash \varphi$ for every $w \in W$. Finally, $\varphi$ is valid in a frame $F = \pair{W, R, I}$, i.e. $F \vDash \varphi$, if $\varphi$ is valid in every model $M$ based on frame $F$.
\end{definition}

In \cite{freire_modality_2024}, we discuss the non-normality of the consequence relation in many-logic modal logic. We also explore novel kinds of frame that incorporate to the accessibility relation the change of background logic; for instance, cases where accessibility is classically increasing or classically decreasing. In \cite{freire_institution_2025}, the many-logic modal logic framework is integrated into the theory of institutions. There, we define relations between models that preserve specific logical properties, enriching the formal specification of systems involving multiple logics. In \cite{freire_lattices_2025}, we address the problem of integrating multiple logics with lattice semantics into a single meaningful lattice that allows for the implementation of the many-logic modal logic framework. We propose a method for constructing a common lattice—based on direct product construction that enables meaningful communication between different logical systems.

In terms of application, \cite{martins_many-logic_2025} employs many-logic modal structures (MLMS) to model information systems that involve diverse logical foundations. These structures use lattices to establish connections between worlds operating in different logics, thereby facilitating the representation of some complex scenarios, including paraconsistent and paracomplete contexts. The framework is illustrated using a six-valued logic of evidence and truth, whose sublattices support a variety of logical contexts.

\section{four-eight logics}\label{sec:four-eight-logics}

In this section, we study a four-valued logic within the framework of a many-valued modal logic, based on the eight-valued Boolean algebra $B_8^\circ$ and an ultrafilter over it. Specifically, we define a special case of a many-logic modal logic by taking $B_8^\circ$ as the underlying lattice $L$, and taking $LAT$ to be the set consisting of the \textbf{three} 4-valued subalgebras of $B_8^\circ$. These subalgebras are shown in Figure~\ref{fig-8boolean}, each represented in a different color. The set of designated values is assumed to be any ultrafilter $U$ of $B_8^\circ$. The elements of each four-valued subalgebra $B$ of $B_8^\circ$ will be written as $1,0,z,-z$, where $\{1,z\}=B\cap U$.

The consider the language $\Lg^{\Box,\circ}$ composed by a denumerably  set $Var$ of variables, the unary connectives  $\neg$, $\Box$ and $\circ$ and the binary connectives $\land$, $\lor$, 
and parentheses, built in the usual way.

We say that $\varphi$ is a semantic consequence (global) of $\Sigma$, in short, $\Sigma \models \varphi$, if any model $M$ of $\Sigma$ is such that $M \vDash \varphi$.

\begin{figure}[H]
    \centering
        
\tikzset{every picture/.style={line width=0.75pt}} %set default line width to 0.75pt        

\begin{tikzpicture}[x=0.75pt,y=0.75pt,yscale=-1,xscale=1]
%uncomment if require: \path (0,300); %set diagram left start at 0, and has height of 300

%Straight Lines [id:da9390929007098849] 
\draw    (261,32) -- (310.1,90.11) ;
\draw [shift={(310.1,90.11)}, rotate = 49.81] [color={rgb, 255:red, 0; green, 0; blue, 0 }  ][fill={rgb, 255:red, 0; green, 0; blue, 0 }  ][line width=0.75]      (0, 0) circle [x radius= 3.35, y radius= 3.35]   ;
\draw [shift={(261,32)}, rotate = 49.81] [color={rgb, 255:red, 0; green, 0; blue, 0 }  ][fill={rgb, 255:red, 0; green, 0; blue, 0 }  ][line width=0.75]      (0, 0) circle [x radius= 3.35, y radius= 3.35]   ;
%Straight Lines [id:da1813804223200528] 
\draw [color={rgb, 255:red, 0; green, 0; blue, 0 }  ,draw opacity=1 ][fill={rgb, 255:red, 0; green, 0; blue, 0 }  ,fill opacity=1 ]   (261,32) -- (210.6,90.11) ;
\draw [shift={(210.6,90.11)}, rotate = 130.94] [color={rgb, 255:red, 0; green, 0; blue, 0 }  ,draw opacity=1 ][fill={rgb, 255:red, 0; green, 0; blue, 0 }  ,fill opacity=1 ][line width=0.75]      (0, 0) circle [x radius= 3.35, y radius= 3.35]   ;
\draw [shift={(261,32)}, rotate = 130.94] [color={rgb, 255:red, 0; green, 0; blue, 0 }  ,draw opacity=1 ][fill={rgb, 255:red, 0; green, 0; blue, 0 }  ,fill opacity=1 ][line width=0.75]      (0, 0) circle [x radius= 3.35, y radius= 3.35]   ;
%Straight Lines [id:da5057452169685153] 
\draw    (261,32) -- (261.1,88.3) ;
\draw [shift={(261.1,88.3)}, rotate = 89.9] [color={rgb, 255:red, 0; green, 0; blue, 0 }  ][fill={rgb, 255:red, 0; green, 0; blue, 0 }  ][line width=0.75]      (0, 0) circle [x radius= 3.35, y radius= 3.35]   ;
\draw [shift={(261,32)}, rotate = 89.9] [color={rgb, 255:red, 0; green, 0; blue, 0 }  ][fill={rgb, 255:red, 0; green, 0; blue, 0 }  ][line width=0.75]      (0, 0) circle [x radius= 3.35, y radius= 3.35]   ;
%Straight Lines [id:da9515686466113575] 
\draw    (260.98,198.12) -- (211.31,140.5) ;
\draw [shift={(211.31,140.5)}, rotate = 229.24] [color={rgb, 255:red, 0; green, 0; blue, 0 }  ][fill={rgb, 255:red, 0; green, 0; blue, 0 }  ][line width=0.75]      (0, 0) circle [x radius= 3.35, y radius= 3.35]   ;
\draw [shift={(260.98,198.12)}, rotate = 229.24] [color={rgb, 255:red, 0; green, 0; blue, 0 }  ][fill={rgb, 255:red, 0; green, 0; blue, 0 }  ][line width=0.75]      (0, 0) circle [x radius= 3.35, y radius= 3.35]   ;
%Straight Lines [id:da6354400889514331] 
\draw    (260.98,198.12) -- (310.8,139.51) ;
\draw [shift={(310.8,139.51)}, rotate = 310.37] [color={rgb, 255:red, 0; green, 0; blue, 0 }  ][fill={rgb, 255:red, 0; green, 0; blue, 0 }  ][line width=0.75]      (0, 0) circle [x radius= 3.35, y radius= 3.35]   ;
\draw [shift={(260.98,198.12)}, rotate = 310.37] [color={rgb, 255:red, 0; green, 0; blue, 0 }  ][fill={rgb, 255:red, 0; green, 0; blue, 0 }  ][line width=0.75]      (0, 0) circle [x radius= 3.35, y radius= 3.35]   ;
%Straight Lines [id:da10773996013952014] 
\draw    (260.98,198.12) -- (260.32,141.83) ;
\draw [shift={(260.32,141.83)}, rotate = 269.33] [color={rgb, 255:red, 0; green, 0; blue, 0 }  ][fill={rgb, 255:red, 0; green, 0; blue, 0 }  ][line width=0.75]      (0, 0) circle [x radius= 3.35, y radius= 3.35]   ;
\draw [shift={(260.98,198.12)}, rotate = 269.33] [color={rgb, 255:red, 0; green, 0; blue, 0 }  ][fill={rgb, 255:red, 0; green, 0; blue, 0 }  ][line width=0.75]      (0, 0) circle [x radius= 3.35, y radius= 3.35]   ;
%Straight Lines [id:da8818861660474588] 
\draw    (210.6,90.11) -- (260.32,141.83) ;
%Straight Lines [id:da2365640135696674] 
\draw    (261.1,88.3) -- (310.82,140.01) ;
%Straight Lines [id:da20508562484642034] 
\draw    (210.6,90.11) -- (211.31,140.5) ;
%Straight Lines [id:da06522942741489413] 
\draw    (261.08,87.79) -- (211.31,140.5) ;
%Straight Lines [id:da2203548086611402] 
\draw    (310.1,90.11) -- (310.82,140.01) ;
%Straight Lines [id:da2830262634373686] 
\draw    (310.1,90.11) -- (260.32,141.83) ;
%Shape: Circle [id:dp5699746468829341] 
\draw  [color={rgb, 255:red, 208; green, 2; blue, 27 }  ,draw opacity=1 ] (202.04,90.11) .. controls (202.04,85.39) and (205.87,81.56) .. (210.6,81.56) .. controls (215.32,81.56) and (219.15,85.39) .. (219.15,90.11) .. controls (219.15,94.84) and (215.32,98.67) .. (210.6,98.67) .. controls (205.87,98.67) and (202.04,94.84) .. (202.04,90.11) -- cycle ;
%Shape: Circle [id:dp32478005193313897] 
\draw  [color={rgb, 255:red, 208; green, 2; blue, 27 }  ,draw opacity=1 ] (302.25,139.51) .. controls (302.25,134.78) and (306.08,130.95) .. (310.8,130.95) .. controls (315.53,130.95) and (319.36,134.78) .. (319.36,139.51) .. controls (319.36,144.23) and (315.53,148.06) .. (310.8,148.06) .. controls (306.08,148.06) and (302.25,144.23) .. (302.25,139.51) -- cycle ;
%Shape: Circle [id:dp9429656507768538] 
\draw  [color={rgb, 255:red, 74; green, 144; blue, 226 }  ,draw opacity=1 ] (252.54,88.3) .. controls (252.54,83.57) and (256.37,79.74) .. (261.1,79.74) .. controls (265.83,79.74) and (269.66,83.57) .. (269.66,88.3) .. controls (269.66,93.02) and (265.83,96.85) .. (261.1,96.85) .. controls (256.37,96.85) and (252.54,93.02) .. (252.54,88.3) -- cycle ;
%Shape: Circle [id:dp7884388470948431] 
\draw  [color={rgb, 255:red, 74; green, 144; blue, 226 }  ,draw opacity=1 ] (251.76,141.83) .. controls (251.76,137.1) and (255.59,133.27) .. (260.32,133.27) .. controls (265.05,133.27) and (268.88,137.1) .. (268.88,141.83) .. controls (268.88,146.55) and (265.05,150.38) .. (260.32,150.38) .. controls (255.59,150.38) and (251.76,146.55) .. (251.76,141.83) -- cycle ;
%Shape: Circle [id:dp706014102347309] 
\draw  [color={rgb, 255:red, 65; green, 117; blue, 5 }  ,draw opacity=1 ] (202.75,140.5) .. controls (202.75,135.77) and (206.58,131.94) .. (211.31,131.94) .. controls (216.03,131.94) and (219.87,135.77) .. (219.87,140.5) .. controls (219.87,145.22) and (216.03,149.05) .. (211.31,149.05) .. controls (206.58,149.05) and (202.75,145.22) .. (202.75,140.5) -- cycle ;
%Shape: Circle [id:dp9891150595447891] 
\draw  [color={rgb, 255:red, 65; green, 117; blue, 5 }  ,draw opacity=1 ] (301.54,90.11) .. controls (301.54,85.39) and (305.37,81.56) .. (310.1,81.56) .. controls (314.82,81.56) and (318.65,85.39) .. (318.65,90.11) .. controls (318.65,94.84) and (314.82,98.67) .. (310.1,98.67) .. controls (305.37,98.67) and (301.54,94.84) .. (301.54,90.11) -- cycle ;

% Text Node
\draw (253.5,3.5) node [anchor=north west][inner sep=0.75pt]   [align=left] {1};
% Text Node
\draw (253.5,208) node [anchor=north west][inner sep=0.75pt]   [align=left] {0};
\end{tikzpicture}
\caption{The Boolean algebra of 8 values and its sub algebras with 4 values.}
    \label{fig-8boolean}
\end{figure}
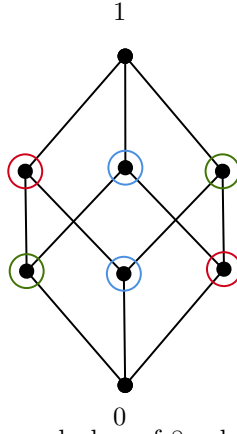

\begin{notation} The following notation will be useful in the sequel.
\begin{itemize}\item 
    We write $\Box \Sigma$ for the set $\{\Box \alpha \mid \alpha \in \Sigma\}$
    and $\circ \Sigma$ for the set $\{\circ \alpha \mid \alpha \in \Sigma\}$
\item     The symbol $\veebar$ will be used for exclusive disjunction (i.e $x \veebar y \cong (x \land \lnot y) \lor (\lnot x \land y)$)
\end{itemize}
\end{notation}

We now begin to formalize the system in a more general and systematic way. Consider the three four-valued sub-Boolean algebras of $B_8$, denoted by $B_1$, $B_2$, and $B_3$. Each of these subalgebras induces a corresponding consequence relation, which we denote by $\vdash_1$, $\vdash_2$, and $\vdash_3$, respectively. Due to the symmetry among the subalgebras, these consequence relations will ultimately coincide. However, for the sake of clarity and generality, we shall begin by analyzing them individually.

The definition will be done recursively:

\begin{enumerate}
    \item Rules and axioms of \B \; logic for the extended language.
    \item $\vdash_1 \Box (\varphi \so \psi) \so (\Box \varphi \so \Box \psi)$ for every $\varphi$ and $\psi$.
    \item $\Box \varphi, \circ \Box \varphi \vdash_1 \Box \circ \varphi$.
    \item $\Diamond \varphi, \Diamond \lnot \varphi \vdash_1 \circ \Box \varphi$
    \item $\lnot \circ \Box \varphi \vdash_1 \Diamond (\varphi \land \lnot \circ \varphi) \veebar \Diamond (\lnot \varphi \land \lnot \circ \varphi)$.
    \item $\Lambda \cup \Box \Gamma \vdash_1 \Box \varphi$ when 
        \begin{enumerate}
            \item $\Lambda \cup \Gamma \cup \circ \Gamma \vdash_2 \varphi \land \circ \varphi$,
            \item $\Lambda \cup \Gamma \cup \circ \Gamma \vdash_3 \varphi \land \circ \varphi$ and
            \item $\Lambda \cup \Gamma \vdash_1 \varphi$.
        \end{enumerate}
\end{enumerate}

    Because $(\vdash_1) = (\vdash_2) = (\vdash_3)$, definitions of $\vdash_2$ and $\vdash_3$ will be done precisely the same way. This symmetry allows us to unite the definitions of $\vdash_1$, $\vdash_2$, and $\vdash_3$ in a single consequence relation $\vdash$:

\begin{enumerate}
    \item\label{clmodal} (CLM) Rules and axioms of \B \, logic for the extended language. (Classical logic for modal language)
    \item\label{kaxiom} (KA) $\vdash \Box (\varphi \so \psi) \so (\Box \varphi \so \Box \psi)$ for every $\varphi$ and $\psi$. (K axiom)
    \item\label{ball-box} (BB) $\Box \varphi, \circ \Box \varphi \vdash \Box \circ \varphi$. (Ball and box axiom)
    \item\label{folse-certainty} (FC) $\Diamond \varphi, \Diamond \lnot \varphi \vdash \circ \Box \varphi$ (False necessity certainty)
    \item\label{existence-axiom} (EA) $\lnot \circ \Box \varphi \vdash \Diamond (\varphi \land \lnot \circ \varphi) \veebar \Diamond (\lnot \varphi \land \lnot \circ \varphi)$. (Existence axiom)
    \item\label{introduction-necessity} (IN) $\Lambda \cup \Box \Gamma \vdash \Box \varphi$ when 
        \begin{enumerate}
            \item $\Lambda \cup \Gamma \cup \circ \Gamma \vdash \varphi \land \circ \varphi$ and
            \item $\Lambda \cup \Gamma \vdash \varphi$.
        \end{enumerate}
        (Introduction of Necessity)
\end{enumerate}

\begin{theorem}\label{ktheorems}
    All theorems of the system K are preserved in $\vdash$.
\end{theorem}

\begin{proof}
The axiom $K$ and all modal instances of classical tautologies are axioms of $\vdash$. Moreover, the necessitation rule is a consequence of the rule (IN), by considering $\Delta$ and $\Gamma$ to be empty sets and the fact $\vdash\varphi$ implies $\vdash \varphi \wedge\circ \varphi$ (by $\wedge$-introduction and (IB)-rule of system \B).
\end{proof}

\begin{definition}
    We say that $\Sigma$ is {\em locally inconsistent} if there are $\alpha_1, \alpha_2, \ldots, \alpha_n \in \Sigma$ such that $\vdash \lnot (\alpha_1 \land \alpha_2 \land \ldots \land \alpha_n)$. We say that $\Sigma$ is {\em locally consistent} if $\Sigma$ is not locally inconsistent.
\end{definition}

\begin{lemma}\label{proofByContradiction}
    Let $\Sigma \cup \{\lnot \varphi\}$ be a set of formulas, if $\Sigma$ is locally consistent and $\Sigma \cup \{\lnot \varphi\}$ is locally inconsistent, then $\Sigma \vdash \varphi$.
\end{lemma}

\begin{proof}
    Since $\Sigma \cup \{\lnot \varphi\}$ is locally inconsistent, there are $\alpha_1, \alpha_2, \ldots, \alpha_n \in \Sigma$ such that $\vdash \lnot (\alpha_1 \land \alpha_2 \land \ldots \land \alpha_n)$. Since $\Sigma$ is locally consistent, one of the $\alpha_i$ must be $\lnot \varphi$. Without loss of generality, assume that $\alpha_n$ is $\lnot \varphi$. From the rules of classical logic, we obtain $\vdash (\alpha_1 \land \alpha_2 \land \ldots \land \alpha_{n-1}) \so \varphi$. Since $\alpha_1, \alpha_2, \ldots, \alpha_n \in \Sigma$, we obtain by \textit{Modus Ponens} that $\Sigma \vdash \varphi$.
\end{proof}

\begin{definition}
    Let $\Sigma$ be a locally consistent set of formulas. We say that $\Lambda$ is a \emph{maximal locally consistent extension of} $\Sigma$ if (i) $\Sigma \subseteq \Lambda$, (ii) $\Lambda$ is locally consistent and (iii) $\Lambda \subsetneq \Pi$ implies that $\Pi$ is locally inconsistent.
\end{definition}

\begin{lemma}\label{worldbivalence}
    Let $\Lambda$ be a maximal locally consistent extension of the locally consistent set $\Sigma$, then, for every formula $\varphi$,
    $$\varphi \in \Lambda \text{ or } \lnot \varphi \in \Lambda$$
\end{lemma}

\begin{proof}
    Suppose $\Lambda \cup \{\varphi\}$ and $\Lambda \cup \{\lnot \varphi\}$ are both locally inconsistent. Then we obtain $\alpha_1, \ldots, \alpha_n, \beta_1, \ldots, \beta_m \in \Sigma$ such that $\vdash (\alpha_1 \land \alpha_2 \land \ldots \land \alpha_n) \so \varphi$ and $\vdash (\beta_1 \land \beta_2 \land \ldots \land \beta_n) \so \lnot \varphi$. From the rules of classical logic we obtain $\vdash \lnot (\alpha_1 \land \alpha_2 \land \ldots \land \alpha_n \land \beta_1 \land \beta_2 \land \ldots \land \beta_n)$ contradicting the fact that $\Lambda$ is locally consistent. 
\end{proof}

\begin{lemma}
    For any locally consistent set $\Sigma$ there is a maximal locally consistent extension $\Lambda$ of $\Sigma$.
\end{lemma}

\begin{proof}
    For this, we need only to apply the common strategy together with the Lemma~\ref{worldbivalence}.
\end{proof}

\begin{definition}
    For a set $\Gamma$, the \emph{canonical model with respect to} $\Gamma$ is the tuple $\pair{M, R, I, v}$ such that
    \begin{enumerate}
        \item $M$ is the collection of all $\pair{i, \Lambda}$ such that $i \in \{1,2,3\}$ and $\Lambda$ is a maximal locally consistent extension of $\Gamma$.
        \item The Boolean algebra of $\pair{i, \Lambda}$ is $B_i$, that is, $I(\pair{i, \Lambda}) = B_i$.
        \item If $w_1, w_2 \in M$ and $I(w_1) = I(w_2)$, then 
        $$w_1 R w_2 \miff \forall \varphi (\Box \varphi \in w_1 \so \varphi \in w_2)$$
        \item If $w_1, w_2 \in M$ and $I(w_1) \neq I(w_2)$, then 
        $$w_1 R w_2 \miff \forall \varphi (\Box \varphi \in w_1 \so \varphi \land \circ \varphi \in w_2)$$
        \item For $w \in M$ such that $I(w) = B_i = \{1, 0, z, -z\}$, the valuation $v_w$ in $w$ is such that (we write $\varphi \in w$ for $\varphi \in \Lambda$ with $w = \pair{i, \Lambda}$):
        \begin{enumerate}
            \item $v_w(X) = 1$ when $X \in w$ and $\circ X \in w$.
            \item $v_w(X) = 0$ when $X \notin w$ and $\circ X \in w$.
            \item $v_w(X) = z$ when $X \in w$ and $\circ X \notin w$.
            \item $v_w(X) = -z$ when $X \notin w$ and $\circ X \notin w$.
        \end{enumerate}
    \end{enumerate}
\end{definition}

\begin{lemma}\label{boxnotin}
    Let $\mathcal{M} = \pair{M, R,I, v}$ be the canonical model with respect to a locally consistent set $\Gamma$. Fix any $w \in M$ and that $\Pi = \{\alpha \mid \Box \alpha \in w\}$. Then 
    $$\Box \varphi \notin w \text{ implies that } \Gamma \cup\Pi \cup \{\lnot \varphi\}\text{ is locally consistent}$$
\end{lemma}

\begin{proof}
    Suppose that $\Box \varphi \notin w$ and that $\Gamma \cup \Pi  \cup \{\lnot \varphi\}$ is locally inconsistent. Then we obtain $\vdash (\gamma_1 \land \gamma_2\land \ldots \land \gamma_m) \to ((\pi_1 \land \pi_2 \land \ldots \land \pi_n) \so \varphi)$ for some $\pi_1, \pi_2, \ldots, \pi_n \in \Pi$ and $\gamma_1, \gamma_2, \ldots, \gamma_m \cup\Gamma$. Since $\Gamma\subset w$,  $((\pi_1 \land \pi_2 \land \ldots \land \pi_n)\so \varphi)\in w$.
    From (IN)-\ref{introduction-necessity} and (KA)-\ref{kaxiom}, $(\Box (\pi_1 \land \pi_2 \land \ldots \land \pi_n) \so \Box \varphi) \in w$. Since $\Box \pi_1, \Box \pi_2, \ldots, \Box \pi_n \in w$, we obtain the contradiction $\Box \varphi \in w$.
\end{proof}

\begin{lemma}\label{valuation_canonical_model}
    Let $\mathcal{M} = \pair{M, R,I, v}$ be the canonical model with respect to a locally consistent set $\Gamma$. Then, for every formula $\varphi$ and every world $w \in M$ such that $I(w) = \{1,0,z,-z\}$
    \begin{enumerate}
        \item $v_w(\varphi) = 1$ iff $\varphi \in w$ and $\circ \varphi \in w$,
        \item $v_w(\varphi) = 0$ iff $\varphi \notin w$ and $\circ \varphi \in w$,
        \item $v_w(\varphi) = z$ iff $\varphi \in w$ and $\circ \varphi \notin w$,
        \item $v_w(\varphi) = -z$ iff $\varphi \notin w$ and $\circ \varphi \notin w$,
    \end{enumerate}
\end{lemma}

In the proof, it will be useful to consider the equivalent list for the negation:
\begin{enumerate}
    \item $v_w(\varphi) \neq 1$ iff $\varphi \notin w$ or $\circ \varphi \notin w$,
    \item $v_w(\varphi) \neq 0$ iff $\varphi \in w$ or $\circ \varphi \notin w$,
    \item $v_w(\varphi) \neq z$ iff $\varphi \notin w$ or $\circ \varphi \in w$,
    \item $v_w(\varphi) \neq -z$ iff $\varphi \in w$ or $\circ \varphi \in w$.
\end{enumerate}

\begin{proof} Fix $w \in M$ with Boolean algebra $\{1,0,a,-a\}$. We prove by induction each item of the lemma. We will skip the induction cases for the compositions with the Boolean operator $\{\lnot, \land, \circ\}$ since the proof is very similar to Lemma~\ref{conservation-values-10-classical}. So consider $\varphi = \Box \psi$. 
    
    Suppose (\textbf{i}) that $\Box \psi \in w$ and $\circ \Box \psi \in w$, but that $v_w(\Box \psi) \neq 1$. 
    Then it exists a $w'$ such that (\textbf{i.a}) $w R w'$, $v_{w'}(\psi) \neq 1$ and $I(w) \neq I(w')$ or (\textbf{i.b}) $w R w'$, $v_{w'}(\psi) \in \{0, -a\}$ and $I(w) = I(w')$. 
    If (\textbf{i.a}) is the case, we conclude $\psi \notin w'$ or $\circ \psi \notin w'$ from IH. From $w R w'$, $I(w) \neq I(w')$ and $\Box \psi \in w$, we obtain $\psi \in w'$ and $\circ \psi \in w'$, contradicting what we concluded in the previous phrase. If (\textbf{i.b}) is the case, we conclude that $\psi \notin w'$. From $w R w'$, $I(w) = I(w')$ and $\Box \psi \in w$, we obtain the contradiction $\psi \in w'$. Now suppose (\textbf{i'}) that $v_w(\Box \psi) = 1$ but that $\Box \psi \notin w$ or $\circ \Box \psi \notin w$. Consider the set $\Pi = \{\alpha \mid \Box \alpha \in w\}$.  
    Suppose $\Pi \cup \{\lnot (\psi \land \circ \psi)\}$ is locally inconsistent. Then from Lemma~\ref{proofByContradiction}, $\Pi \vdash \psi \land \circ \psi$ and, thus, using (IN)-\ref{introduction-necessity}, we obtain $\Box \Pi \vdash \Box \psi$. Once $\Box \Pi \subseteq w$, we have that $\Box \psi \in w$. Hence $\circ \Box \psi \notin w$ and, consequently, that $\lnot \circ \Box \psi \in w$. Clearly, $\Diamond (\neg \psi \land \lnot \circ \psi) \notin w$, thus, from (EA)-\ref{existence-axiom}, it follows that $\Diamond (\psi \land \lnot \circ \psi) \in w$. Equivalently, we have that $\Box (\lnot \psi \lor \circ \psi) \notin w$. From Lemma~\ref{boxnotin}, $\Gamma \cup\Pi \cup \{\psi \land \lnot \circ \psi\}$ is locally consistent. This is absurd since $\Pi \vdash \circ \psi$. We thus conclude that $\Gamma \cup\Pi \cup \{\lnot (\psi \land \circ \psi)\}$ is locally consistent. Consider a world $w' \in M$ such that $w' \supseteq \Gamma \cup\Pi \cup \{\lnot (\psi \land \circ \psi)\}$. Then $w R w'$ and, from IH, $v_{w'}(\psi) \neq 1$. Consequently, we obtain the absurd $v_w(\Box \psi) \neq 1$.

    Suppose (\textbf{ii}) that $\Box \psi \notin w$ and $\circ \Box \psi \in w$, but that $v_w(\Box \psi) \neq 0$. Because of Lemma~\ref{boxnotin}, $\Gamma \cup\Pi \cup \{\lnot \psi\}$ is locally consistent for $\Pi = \{\alpha \mid \Box \alpha \in w\}$. 
    
    Thence we obtain a world $w'$ such that $I(w') \neq I(w)$ and $w'$ contains $\Pi \cup \{\lnot \psi\}$. Naturally, $w R w'$; also, from IH, $v_{w'}(\psi) \neq 1$. In this case, we obtain $v_w(\Box \psi) = 0$.
    
    Now suppose (\textbf{ii'}) that $v_w(\Box \psi) = 0$, but that $\Box \psi \in w$ or $\circ \Box \psi \notin w$. The fact that $v_w(\Box \psi) = 0$ give us two possible scenarios: (\textbf{ii'.a}) there is $w'$ such that $w R w'$ and $v_{w'}(\psi) = 0$ or (\textbf{ii'.b}) there are $w_1$ and $w_2$ such that $w R w_1$, $w R w_2$, $I(w_1) = I(w_2) = I(w)$, $v_{w_1} = z$ and $v_{w_2} = -z$. From IH, if we have (\textbf{ii'.a}), we conclude $\psi \notin w'$, which contradicts $w R w'$ and the fact that $\Box \psi \in w$. If we have (\textbf{ii'.b}), the fact that $v_{w_2} = -z$ give us from IH that $\psi \notin w_2$. This again contradicts $w R w'$ and $\Box \psi \in w$.

    Suppose (\textbf{iii}) that $\Box \psi \in w$ and $\circ \Box \psi \notin w$. From (EA)-\ref{existence-axiom}, we obtain that $\Diamond (\psi \land \lnot \circ \psi) \in w$ (as before, equivalently, $\Box (\lnot \psi \lor \circ \psi) \notin w$). From Lemma~\ref{boxnotin}, $\Gamma \cup\Pi \cup \{\psi \land \lnot \circ \psi\}$ is locally consistent for $\Pi = \{\alpha \mid \Box \alpha \in w\}$. So there is a world $w'$ such that $I(w') = I(w)$, $w'$ contains $\Pi \cup \{\psi \land \lnot \circ \psi\}$ and $w R w'$. From IH, $v_{w'}(\psi) = z$. Now suppose $w R w_1$. Then we obtain $\psi \in w_1$ and, consequently, $v_{w_1}(\psi) \in \{y, 1\}$ from IH ($y$ being the value different from $1$ in the filter of $w_1$'s algebra). If $I(w) \neq I(w_1)$ and $v_{w_1}(\psi) = y$, then $\lnot \circ \psi \in w_1$. But this contradicts $w R w_1$. Therefore $I(w) = I(w_1)$ or $v_{w_1}(\psi) = 1$. Since we have $v_{w'}(\psi) = z$ and all other values of the accessibles are either $1$ or $z$ we conclude that $v_w(\Box \psi) = z$. Now suppose (\textbf{iii'}) $v_w(\Box \psi) = z$. If $\Box \psi \notin w$, then $\Gamma \cup\Pi \cup \{\lnot \psi\}$ is consistent where $\Pi = \{\alpha \mid \Box \alpha \in w\}$ from Lemma~\ref{boxnotin}. Then there is $w'$ with $I(w') \neq I(w)$ such that $\psi \notin w'$. From IH, $v_{w'}(\psi) \neq 1$; thus $v_w(\Box \psi) = 0$, absurd. So $\Box \psi \in w$. Assume $\circ \Box \psi \in w$; from (BB)-\ref{ball-box}, $\Box \circ \psi$. From $v_w(\Box \psi) = z$, there is $w'$ such that $v_{w'}(\psi) = z$ and $w R w'$. From IH, $\psi \in w'$ and $\lnot \circ \psi \in w'$. Since $w R w'$, $I(w') = I(w)$ and $\Box \psi \in w$, $\psi \land \circ \psi \in w'$ which is absurd. 

    The proof of case (\textbf{iv}) is similar to the one given for (\textbf{iii}).  We leave it for the reader.

\end{proof}

\begin{remark}
    Notice that the proof of Theorem~\ref{valuation_canonical_model} does not require that we have worlds for all kinds of world. 
    More concretely, the proof never requires us to consider $I(w) \neq I(w') \neq I(w'') \neq I(w)$.
    Consequently, we may have Theorem~\ref{valuation_canonical_model} for the canonical model $M$ extending the locally consistent $\Lambda$ using only worlds with $B_1$ and $B_2$ instead of having worlds with $B_1$, $B_2$ and $B_3$.
\end{remark}

\begin{corollary}\label{valuation_canonical_model_2lattices}
    Let $\mathcal{M} = \pair{M, R, v}$ be the canonical model with respect to a locally consistent set $\Gamma$. Let $M' = \{w \in M \mid I(w) = B_1 \text{ or } I(w) = B_2\}$ and $\mathcal{M}' = \pair{M', R|M', v|M'}$.
    
    Then, for every formula $\varphi$ and every world $w \in M'$
    \begin{enumerate}
        \item $v_w(\varphi) = 1$ iff $\varphi \in w$ and $\circ \varphi \in w$,
        \item $v_w(\varphi) = 0$ iff $\varphi \notin w$ and $\circ \varphi \in w$,
        \item $v_w(\varphi) = z$ iff $\varphi \in w$ and $\circ \varphi \notin w$,
        \item $v_w(\varphi) = -z$ iff $\varphi \notin w$ and $\circ \varphi \notin w$,
    \end{enumerate}
    where $z$ is a value in $w$'s algebra different from $1$ and from $0$.
\end{corollary}

This corollary will be used in the next section. This will allow us to prove that some frame properties are not characterizable. We do not need to use this to demonstrate soundness and completeness.

\begin{theorem}\label{sound-complete-modal}
     For $\Sigma \cup \{\varphi\}$ a set of $\mathcal{L}^{\Box,\circ}$-formulas,
     $$\Sigma \vdash \varphi \text{ if, and only if, } \Sigma \vDash \varphi$$
\end{theorem}

\begin{proof}
\quad

({\em Proof of soundness}) It is worth noting that rule (IN) creates an additional difficulty since the proof of a given statement $\Box \varphi$ depends on the proof of other
statements, namely $\varphi$ and $\varphi \wedge \circ \varphi$) from different sets of premisses. The proof follows the same argument of the proof of soundness in \cite{freire_modality_2024} by induction over the level of the formulas in a $\vdash$ proof.

({\em Proof of completeness}) Suppose $\Sigma \nvdash \varphi$. Then, by Lemma \ref{proofByContradiction} $\Sigma \cup \{\lnot \varphi\}$ is locally consistent. Let $\mathcal{M} = \pair{M, R, v}$ be the canonical model with respect to $\Sigma \cup \{\lnot \varphi\}$. Hence, for any $w\in M$, $\varphi\notin w$. By Lemma \ref{valuation_canonical_model}, $v_w(\varphi)=0$ or $v_w(\varphi)=-z$. Therefore, $\Sigma \not \vDash \varphi$.
\end{proof}

\section{Characterization of frames in the (8-4) modal logic}\label{sec_prop}

A direct consequence of \cref{ktheorems} is that the four-eight modal logic produces the same local consequence relation as the system K. However, the global consequence relation differs. As a result, we can anticipate that the frame characterization will behave differently within this new framework. In this section, we explore three key observations in this regard: (a) certain properties, such as transitivity and reflexivity, are still characterized by the same formulas as in K; (b) other properties may require more intricate formulas for proper characterization; and (c) new forms of characterization become possible, allowing us to account for more nuanced variations in the accessibility relation in light of the change in the base lattice.

\begin{notation}
    We use the symbol $\Vdash$ for the global consequence relation of system K. Moreover, we use $\vDash$ for the global consequence relation developed in \cref{sec:four-eight-logics}.
\end{notation}

\begin{theorem}\label{one-direction-preservation}
    Let $P$ be a property of frames, and let $\varphi$ be a $\cl^\Box$-formula such that, for every traditional frame $F$, $F \Vdash \varphi$ if and only if $F$ is $P$. In other words, $P$ is characterized by the formula $\varphi$ in the K system.
    
    In this case, if a $F$ is a (8-4)-frame\footnote{By a (8-4)-frame we mean a many-logic modal frame over the set of the three 4-valued sub-Boolean algebras of $B_8^\circ$.} such that $F \vDash \varphi$, then $F$ is $P$.
\end{theorem}

\begin{proof}
    Consider $P$, $\varphi$ satisfying the conditions of the theorem. Let $F$ be a frame such that $F \vDash \varphi$. Since every traditional model with frame $F$ is also a model in which we restrict the attributions to the values $0$ and $1$, we obtain $F \Vdash \varphi$. Consequently, $F$ is $P$.
\end{proof}

We shall now observe that the converse relation `$F$ is $P$ implies $F \vDash \varphi$' holds for some traditional properties and does not hold for some others.

\begin{theorem}
    $F$ is reflexive if and only if $F \vDash \Box \varphi \so \varphi$.
\end{theorem}

\begin{proof}
    Suppose $F$ is reflexive and that $F \nvDash \Box \varphi \so \varphi$. There is a model $M$ with frame $F$ in which some world $w$ is such that $w \nvDash \Box \varphi \so \varphi$. Therefore, $w \vDash \Box \varphi$ and $w \nvDash \varphi$. Since $F$ is reflexive, $w$ accesses $w$ and so $v_w(\Box \varphi) \leq v_w(\varphi)$. Consequently, $w \nvDash \Box \varphi$.
\end{proof}

\begin{theorem}
    $F$ is transitive if and only if $F \vDash \Box \varphi \so \Box \Box \varphi$.
\end{theorem}

\begin{proof}
    In a model $M$ with transitive frame $F$, suppose a world $w$ such that $w \vDash \Box \varphi$. Since $F$ is transitive, $\{u \mid \exists w' w R w' \land w' R u\} \subseteq \{u \mid w R u\}$. Therefore, $v_w(\Box \varphi) \leq v_w(\Box \Box \varphi)$ and so $w \vDash \Box \Box \varphi$. 
\end{proof}

\begin{theorem}\label{Euc_no}
    $F$ is Euclidean is \textbf{not} equivalent to $F \vDash \Diamond \varphi \so \Box \Diamond \varphi$.
\end{theorem}

\begin{proof}
    Let $w$, $u$, $u'$ be such that all of them can access each other and themselves. Furthermore, let $L_1 = \{1, 0, a, -a\}$ be the Boolean algebra of $w$ and $u$, and let $L_2 = \{1, 0, b, -b\}$ be the Boolean algebra of $u'$. Consider the valuation $v$ such that $v_w(p) = 0$, $v_u(p) = 0$ and $v_{u'}(p) = b$. Notice that $v_w(\Diamond X) = - \bigwedge \{-0, -0, (-b)^A\} = 1$. Thence $w \vDash \Diamond X$. Similarly, $v_u(\Diamond X) = - \bigwedge \{-0, -0, (-b)^A\} = 1$ and $v_{u'}(\Diamond X) = - \bigwedge \{-0, -0, -b\} = b$. Then $v_w(\Box \Diamond X) = \bigwedge \{1, 1, (b)^A\} = 0$ and, consequently, $w \nvDash \Box \Diamond X$.   
\end{proof}

The problem is deeper in reality. It is not clear whether we can characterize the Euclidean property in the basic language $\mathcal{L}$. However, we can use the extended language to recover the characterization of the Euclidean frames.

\begin{theorem}\label{Euc_yes} $F$ is Euclidean if, and only if, $F \vDash \Diamond \circ \varphi \so \Box \Diamond \circ \varphi$ (we call this axiom $5^\circ$).
\end{theorem}

\begin{proof}
    Suppose $F$ is Euclidean and that there is a world $w$ in a model $M$ with frame $F$ such that $w \nvDash \Diamond \circ \varphi \so \Box \Diamond \circ \varphi$. Because $\circ \varphi$ can only have values $0$ or $1$ in every world, $v_w(\Diamond \circ \varphi) = 1$ and $v_w(\Box \Diamond \circ \varphi) = 0$. So, there are $u$ such that $v_u(\circ \varphi) = 1$ and $u'$ such that $v_{u'}(\Diamond \circ \varphi) = 0$ where both $u$ and $u'$ are accessible to $w$. Since $F$ is Euclidean, $u'$ accesses $u$, so we obtain the absurd $v_u(\circ \varphi) = 0$.
\end{proof}

The fact that we have a harder time finding the characterization of Euclidean frames in our alternative classical modal logic is twofold. First, it shows a nuanced distinction between the assumptions behind transitive and Euclidean properties in traditional modal logic. The traditional characterization of transitivity resists even if we jump to a more skeptical classical semantics, while the traditional Euclidean characterization depends on the more credulous take on the validities of accessed worlds.

\begin{theorem}\label{two-direction-preservation}
    Let $P$ be a property of frames and $\varphi$ a $\cl^\Box$-formula such that, for every traditional frame $F$, 
    $F \Vdash \varphi$ if, and only if, $F$ is $P$. 
    Let $\varphi'$ be the replacement of every variable $x$ in $\varphi$ for $\circ x$.
    Then $F \vDash \varphi'$ if and only if $F$ is $P$.
\end{theorem}

\begin{proof}
   The fact that $F \vDash \varphi'$ implies $F$ is $P$ results from the direct application of \cref{one-direction-preservation}. 

   Suppose $F$ has the property $P$. And that $F \nvDash \varphi'$. Consequently, there is a model $M$ with frame $F$ and valuation $v$ such that $M \nvDash \varphi'$. Now, consider the model $N$ with frame $F$ and valuation $s$ such that
   $$s_w(p) = \begin{cases}
       1, v_w(\circ p) = 1\\
       0, v_w(\circ p) = 0
   \end{cases}$$
   Note that $N$ only attributes values $1$ and $0$ to the propositional variables. Therefore, $N$ is equivalent to a K model. As a result, $N \Vdash \varphi$. A simple inductive argument (left to the reader) establishes that $N \Vdash \psi$ if, and only if, $M \vDash \psi'$. So we obtain the contradiction $M \vDash \varphi'$.

\end{proof}

This result shows that every property of traditional modal logic can be characterized in $\mathcal{L}^{\Box,\circ}$ within the proposed framework. The next step is to determine whether we can also characterize additional properties, those that concern not only the dynamics of accessibility, but also the dynamics of accessibility in conjunction with changes in the underlying lattice. In this direction, consider the following property:

\begin{definition}
    A frame $F$ is out of the bubble when, for every $w$ in $F$, if $w$ accesses a world, then it accesses a world $w'$ such that $I(w) \neq I(w')$. In other words, every world that accesses at least one world also accesses a world that does not have the same base lattice.
\end{definition}

Consider these examples:

\begin{figure}[H]
    \centering
    \tikzset{every picture/.style={line width=0.75pt}} %set default line width to 0.75pt        
    \begin{tikzpicture}[x=0.75pt,y=0.75pt,yscale=-1,xscale=1]
%uncomment if require: \path (0,300); %set diagram left start at 0, and has height of 300

%Shape: Circle [id:dp6273937196242955] 
\draw  [fill={rgb, 255:red, 0; green, 0; blue, 0 }  ,fill opacity=1 ] (95,100.5) .. controls (95,97.46) and (97.46,95) .. (100.5,95) .. controls (103.54,95) and (106,97.46) .. (106,100.5) .. controls (106,103.54) and (103.54,106) .. (100.5,106) .. controls (97.46,106) and (95,103.54) .. (95,100.5) -- cycle ;
%Shape: Circle [id:dp32084777611716175] 
\draw   (53,139.5) .. controls (53,136.46) and (55.46,134) .. (58.5,134) .. controls (61.54,134) and (64,136.46) .. (64,139.5) .. controls (64,142.54) and (61.54,145) .. (58.5,145) .. controls (55.46,145) and (53,142.54) .. (53,139.5) -- cycle ;
%Straight Lines [id:da17563675500013964] 
\draw    (58.5,139.5) -- (100.5,100.5) ;
\draw [shift={(83.16,116.6)}, rotate = 137.12] [fill={rgb, 255:red, 0; green, 0; blue, 0 }  ][line width=0.08]  [draw opacity=0] (10.72,-5.15) -- (0,0) -- (10.72,5.15) -- (7.12,0) -- cycle    ;
%Shape: Circle [id:dp34274795636653577] 
\draw   (139.5,176) .. controls (139.5,172.96) and (141.96,170.5) .. (145,170.5) .. controls (148.04,170.5) and (150.5,172.96) .. (150.5,176) .. controls (150.5,179.04) and (148.04,181.5) .. (145,181.5) .. controls (141.96,181.5) and (139.5,179.04) .. (139.5,176) -- cycle ;
%Shape: Circle [id:dp7756564585688032] 
\draw   (54,198.5) .. controls (54,195.46) and (56.46,193) .. (59.5,193) .. controls (62.54,193) and (65,195.46) .. (65,198.5) .. controls (65,201.54) and (62.54,204) .. (59.5,204) .. controls (56.46,204) and (54,201.54) .. (54,198.5) -- cycle ;
%Shape: Circle [id:dp8361231569947132] 
\draw   (113,142.5) .. controls (113,139.46) and (115.46,137) .. (118.5,137) .. controls (121.54,137) and (124,139.46) .. (124,142.5) .. controls (124,145.54) and (121.54,148) .. (118.5,148) .. controls (115.46,148) and (113,145.54) .. (113,142.5) -- cycle ;
%Straight Lines [id:da07673124133477582] 
\draw    (58.5,139.5) -- (118.5,142.5) ;
\draw [shift={(93.49,141.25)}, rotate = 182.86] [fill={rgb, 255:red, 0; green, 0; blue, 0 }  ][line width=0.08]  [draw opacity=0] (10.72,-5.15) -- (0,0) -- (10.72,5.15) -- (7.12,0) -- cycle    ;
%Straight Lines [id:da21957312335714962] 
\draw    (58.5,139.5) -- (59.5,198.5) ;
\draw [shift={(59.08,174)}, rotate = 269.03] [fill={rgb, 255:red, 0; green, 0; blue, 0 }  ][line width=0.08]  [draw opacity=0] (10.72,-5.15) -- (0,0) -- (10.72,5.15) -- (7.12,0) -- cycle    ;
%Straight Lines [id:da7329043817894998] 
\draw    (118.5,142.5) -- (145,176) ;
\draw [shift={(134.85,163.17)}, rotate = 231.65] [fill={rgb, 255:red, 0; green, 0; blue, 0 }  ][line width=0.08]  [draw opacity=0] (10.72,-5.15) -- (0,0) -- (10.72,5.15) -- (7.12,0) -- cycle    ;
%Straight Lines [id:da062133571535096355] 
\draw    (118.5,142.5) -- (100.5,100.5) ;
\draw [shift={(107.53,116.9)}, rotate = 66.8] [fill={rgb, 255:red, 0; green, 0; blue, 0 }  ][line width=0.08]  [draw opacity=0] (10.72,-5.15) -- (0,0) -- (10.72,5.15) -- (7.12,0) -- cycle    ;
%Shape: Circle [id:dp044872647134939325] 
\draw  [fill={rgb, 255:red, 0; green, 0; blue, 0 }  ,fill opacity=1 ] (97,199.5) .. controls (97,196.46) and (99.46,194) .. (102.5,194) .. controls (105.54,194) and (108,196.46) .. (108,199.5) .. controls (108,202.54) and (105.54,205) .. (102.5,205) .. controls (99.46,205) and (97,202.54) .. (97,199.5) -- cycle ;
%Straight Lines [id:da6502355451178218] 
\draw    (59.5,198.5) -- (102.5,199.5) ;
\draw [shift={(86,199.12)}, rotate = 181.33] [fill={rgb, 255:red, 0; green, 0; blue, 0 }  ][line width=0.08]  [draw opacity=0] (10.72,-5.15) -- (0,0) -- (10.72,5.15) -- (7.12,0) -- cycle    ;
%Straight Lines [id:da5501950511651206] 
\draw    (102.5,199.5) -- (145,176) ;
\draw [shift={(128.13,185.33)}, rotate = 151.06] [fill={rgb, 255:red, 0; green, 0; blue, 0 }  ][line width=0.08]  [draw opacity=0] (10.72,-5.15) -- (0,0) -- (10.72,5.15) -- (7.12,0) -- cycle    ;
%Shape: Circle [id:dp37718665890332714] 
\draw  [fill={rgb, 255:red, 0; green, 0; blue, 0 }  ,fill opacity=1 ] (264,100.5) .. controls (264,97.46) and (266.46,95) .. (269.5,95) .. controls (272.54,95) and (275,97.46) .. (275,100.5) .. controls (275,103.54) and (272.54,106) .. (269.5,106) .. controls (266.46,106) and (264,103.54) .. (264,100.5) -- cycle ;
%Shape: Circle [id:dp21635978943163037] 
\draw   (222,139.5) .. controls (222,136.46) and (224.46,134) .. (227.5,134) .. controls (230.54,134) and (233,136.46) .. (233,139.5) .. controls (233,142.54) and (230.54,145) .. (227.5,145) .. controls (224.46,145) and (222,142.54) .. (222,139.5) -- cycle ;
%Straight Lines [id:da33094906962136217] 
\draw    (227.5,139.5) -- (269.5,100.5) ;
\draw [shift={(252.16,116.6)}, rotate = 137.12] [fill={rgb, 255:red, 0; green, 0; blue, 0 }  ][line width=0.08]  [draw opacity=0] (10.72,-5.15) -- (0,0) -- (10.72,5.15) -- (7.12,0) -- cycle    ;
%Shape: Circle [id:dp17289999853769822] 
\draw   (308.5,176) .. controls (308.5,172.96) and (310.96,170.5) .. (314,170.5) .. controls (317.04,170.5) and (319.5,172.96) .. (319.5,176) .. controls (319.5,179.04) and (317.04,181.5) .. (314,181.5) .. controls (310.96,181.5) and (308.5,179.04) .. (308.5,176) -- cycle ;
%Shape: Circle [id:dp49918568469120483] 
\draw   (223,198.5) .. controls (223,195.46) and (225.46,193) .. (228.5,193) .. controls (231.54,193) and (234,195.46) .. (234,198.5) .. controls (234,201.54) and (231.54,204) .. (228.5,204) .. controls (225.46,204) and (223,201.54) .. (223,198.5) -- cycle ;
%Shape: Circle [id:dp17908602874595447] 
\draw   (282,142.5) .. controls (282,139.46) and (284.46,137) .. (287.5,137) .. controls (290.54,137) and (293,139.46) .. (293,142.5) .. controls (293,145.54) and (290.54,148) .. (287.5,148) .. controls (284.46,148) and (282,145.54) .. (282,142.5) -- cycle ;
%Straight Lines [id:da3487303309748021] 
\draw    (227.5,139.5) -- (287.5,142.5) ;
\draw [shift={(262.49,141.25)}, rotate = 182.86] [fill={rgb, 255:red, 0; green, 0; blue, 0 }  ][line width=0.08]  [draw opacity=0] (10.72,-5.15) -- (0,0) -- (10.72,5.15) -- (7.12,0) -- cycle    ;
%Straight Lines [id:da1396505047430886] 
\draw    (227.5,139.5) -- (228.5,198.5) ;
\draw [shift={(228.08,174)}, rotate = 269.03] [fill={rgb, 255:red, 0; green, 0; blue, 0 }  ][line width=0.08]  [draw opacity=0] (10.72,-5.15) -- (0,0) -- (10.72,5.15) -- (7.12,0) -- cycle    ;
%Straight Lines [id:da18487266722817242] 
\draw    (287.5,142.5) -- (314,176) ;
\draw [shift={(303.85,163.17)}, rotate = 231.65] [fill={rgb, 255:red, 0; green, 0; blue, 0 }  ][line width=0.08]  [draw opacity=0] (10.72,-5.15) -- (0,0) -- (10.72,5.15) -- (7.12,0) -- cycle    ;
%Shape: Circle [id:dp598029832378986] 
\draw  [fill={rgb, 255:red, 0; green, 0; blue, 0 }  ,fill opacity=1 ] (266,199.5) .. controls (266,196.46) and (268.46,194) .. (271.5,194) .. controls (274.54,194) and (277,196.46) .. (277,199.5) .. controls (277,202.54) and (274.54,205) .. (271.5,205) .. controls (268.46,205) and (266,202.54) .. (266,199.5) -- cycle ;
%Straight Lines [id:da37387200021709566] 
\draw    (228.5,198.5) -- (271.5,199.5) ;
\draw [shift={(255,199.12)}, rotate = 181.33] [fill={rgb, 255:red, 0; green, 0; blue, 0 }  ][line width=0.08]  [draw opacity=0] (10.72,-5.15) -- (0,0) -- (10.72,5.15) -- (7.12,0) -- cycle    ;
%Straight Lines [id:da47029456440068074] 
\draw    (271.5,199.5) -- (314,176) ;
\draw [shift={(297.13,185.33)}, rotate = 151.06] [fill={rgb, 255:red, 0; green, 0; blue, 0 }  ][line width=0.08]  [draw opacity=0] (10.72,-5.15) -- (0,0) -- (10.72,5.15) -- (7.12,0) -- cycle    ;
%Shape: Circle [id:dp7132950476990592] 
\draw  [fill={rgb, 255:red, 0; green, 0; blue, 0 }  ,fill opacity=1 ] (445,99.5) .. controls (445,96.46) and (447.46,94) .. (450.5,94) .. controls (453.54,94) and (456,96.46) .. (456,99.5) .. controls (456,102.54) and (453.54,105) .. (450.5,105) .. controls (447.46,105) and (445,102.54) .. (445,99.5) -- cycle ;
%Shape: Circle [id:dp06484345177918294] 
\draw   (403,138.5) .. controls (403,135.46) and (405.46,133) .. (408.5,133) .. controls (411.54,133) and (414,135.46) .. (414,138.5) .. controls (414,141.54) and (411.54,144) .. (408.5,144) .. controls (405.46,144) and (403,141.54) .. (403,138.5) -- cycle ;
%Straight Lines [id:da450052837473657] 
\draw    (408.5,138.5) -- (450.5,99.5) ;
\draw [shift={(433.16,115.6)}, rotate = 137.12] [fill={rgb, 255:red, 0; green, 0; blue, 0 }  ][line width=0.08]  [draw opacity=0] (10.72,-5.15) -- (0,0) -- (10.72,5.15) -- (7.12,0) -- cycle    ;
%Shape: Circle [id:dp17279652486360242] 
\draw  [fill={rgb, 255:red, 0; green, 0; blue, 0 }  ,fill opacity=1 ] (489.5,175) .. controls (489.5,171.96) and (491.96,169.5) .. (495,169.5) .. controls (498.04,169.5) and (500.5,171.96) .. (500.5,175) .. controls (500.5,178.04) and (498.04,180.5) .. (495,180.5) .. controls (491.96,180.5) and (489.5,178.04) .. (489.5,175) -- cycle ;
%Shape: Circle [id:dp5580792212070171] 
\draw   (404,197.5) .. controls (404,194.46) and (406.46,192) .. (409.5,192) .. controls (412.54,192) and (415,194.46) .. (415,197.5) .. controls (415,200.54) and (412.54,203) .. (409.5,203) .. controls (406.46,203) and (404,200.54) .. (404,197.5) -- cycle ;
%Shape: Circle [id:dp5656993622903743] 
\draw   (463,141.5) .. controls (463,138.46) and (465.46,136) .. (468.5,136) .. controls (471.54,136) and (474,138.46) .. (474,141.5) .. controls (474,144.54) and (471.54,147) .. (468.5,147) .. controls (465.46,147) and (463,144.54) .. (463,141.5) -- cycle ;
%Straight Lines [id:da6335193794833999] 
\draw    (408.5,138.5) -- (468.5,141.5) ;
\draw [shift={(443.49,140.25)}, rotate = 182.86] [fill={rgb, 255:red, 0; green, 0; blue, 0 }  ][line width=0.08]  [draw opacity=0] (10.72,-5.15) -- (0,0) -- (10.72,5.15) -- (7.12,0) -- cycle    ;
%Straight Lines [id:da5012068388476875] 
\draw    (408.5,138.5) -- (409.5,197.5) ;
\draw [shift={(409.08,173)}, rotate = 269.03] [fill={rgb, 255:red, 0; green, 0; blue, 0 }  ][line width=0.08]  [draw opacity=0] (10.72,-5.15) -- (0,0) -- (10.72,5.15) -- (7.12,0) -- cycle    ;
%Straight Lines [id:da22156442461919723] 
\draw    (468.5,141.5) -- (495,175) ;
\draw [shift={(484.85,162.17)}, rotate = 231.65] [fill={rgb, 255:red, 0; green, 0; blue, 0 }  ][line width=0.08]  [draw opacity=0] (10.72,-5.15) -- (0,0) -- (10.72,5.15) -- (7.12,0) -- cycle    ;
%Straight Lines [id:da6804642783350874] 
\draw    (468.5,141.5) -- (450.5,99.5) ;
\draw [shift={(457.53,115.9)}, rotate = 66.8] [fill={rgb, 255:red, 0; green, 0; blue, 0 }  ][line width=0.08]  [draw opacity=0] (10.72,-5.15) -- (0,0) -- (10.72,5.15) -- (7.12,0) -- cycle    ;
%Shape: Circle [id:dp4253029797300094] 
\draw  [fill={rgb, 255:red, 0; green, 0; blue, 0 }  ,fill opacity=1 ] (447,198.5) .. controls (447,195.46) and (449.46,193) .. (452.5,193) .. controls (455.54,193) and (458,195.46) .. (458,198.5) .. controls (458,201.54) and (455.54,204) .. (452.5,204) .. controls (449.46,204) and (447,201.54) .. (447,198.5) -- cycle ;
%Straight Lines [id:da26295417617593597] 
\draw    (409.5,197.5) -- (452.5,198.5) ;
\draw [shift={(436,198.12)}, rotate = 181.33] [fill={rgb, 255:red, 0; green, 0; blue, 0 }  ][line width=0.08]  [draw opacity=0] (10.72,-5.15) -- (0,0) -- (10.72,5.15) -- (7.12,0) -- cycle    ;
%Straight Lines [id:da6951099179041051] 
\draw    (452.5,198.5) -- (495,175) ;
\draw [shift={(478.13,184.33)}, rotate = 151.06] [fill={rgb, 255:red, 0; green, 0; blue, 0 }  ][line width=0.08]  [draw opacity=0] (10.72,-5.15) -- (0,0) -- (10.72,5.15) -- (7.12,0) -- cycle    ;
%Straight Lines [id:da312554747857563] 
\draw  [dash pattern={on 0.84pt off 2.51pt}]  (179,52) -- (180,210.5) ;
%Straight Lines [id:da9467952111840804] 
\draw  [dash pattern={on 0.84pt off 2.51pt}]  (371,52) -- (372,210.5) ;

% Text Node
\draw (41,121.4) node [anchor=north west][inner sep=0.75pt]    {$a$};
% Text Node
\draw (36,189.4) node [anchor=north west][inner sep=0.75pt]    {$b$};
% Text Node
\draw (96,212.4) node [anchor=north west][inner sep=0.75pt]    {$c$};
% Text Node
\draw (152.5,179.4) node [anchor=north west][inner sep=0.75pt]    {$d$};
% Text Node
\draw (129,127.4) node [anchor=north west][inner sep=0.75pt]    {$e$};
% Text Node
\draw (94,73.4) node [anchor=north west][inner sep=0.75pt]    {$f$};
% Text Node
\draw (48,39) node [anchor=north west][inner sep=0.75pt]   [align=left] {$\displaystyle F_{1}$: Out of bubble};
% Text Node
\draw (213,41) node [anchor=north west][inner sep=0.75pt]   [align=left] {$\displaystyle F_{2}$:\textbf{ Not} Out of bubble};
% Text Node
\draw (388,42) node [anchor=north west][inner sep=0.75pt]   [align=left] {$\displaystyle F_{3}$:\textbf{ Not} Out of bubble};
% Text Node
\draw (210,121.4) node [anchor=north west][inner sep=0.75pt]    {$a$};
% Text Node
\draw (205,189.4) node [anchor=north west][inner sep=0.75pt]    {$b$};
% Text Node
\draw (265,212.4) node [anchor=north west][inner sep=0.75pt]    {$c$};
% Text Node
\draw (321.5,179.4) node [anchor=north west][inner sep=0.75pt]    {$d$};
% Text Node
\draw (298,127.4) node [anchor=north west][inner sep=0.75pt]    {$e$};
% Text Node
\draw (263,73.4) node [anchor=north west][inner sep=0.75pt]    {$f$};
% Text Node
\draw (391,120.4) node [anchor=north west][inner sep=0.75pt]    {$a$};
% Text Node
\draw (386,188.4) node [anchor=north west][inner sep=0.75pt]    {$b$};
% Text Node
\draw (446,211.4) node [anchor=north west][inner sep=0.75pt]    {$c$};
% Text Node
\draw (502.5,178.4) node [anchor=north west][inner sep=0.75pt]    {$d$};
% Text Node
\draw (479,126.4) node [anchor=north west][inner sep=0.75pt]    {$e$};
% Text Node
\draw (444,72.4) node [anchor=north west][inner sep=0.75pt]    {$f$};

\end{tikzpicture}
    \caption{Example of out of bubble frame}
    \label{fig-example}
\end{figure}
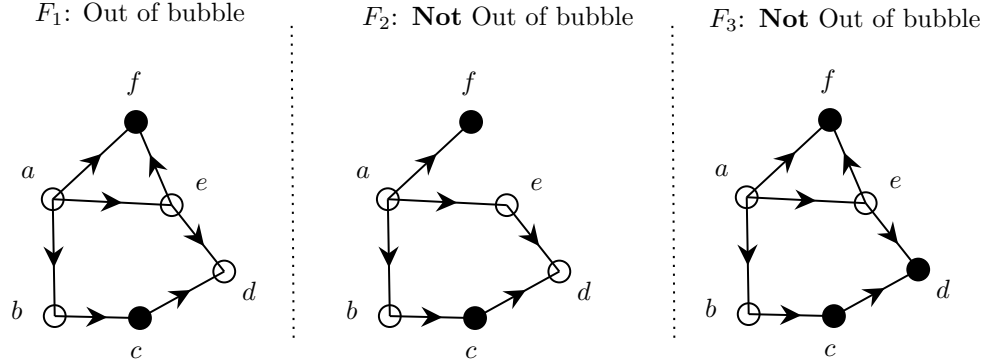

The filled circles represent worlds operating under one logic, while the empty circles represent worlds operating under another logic. Observe that, in the frame $F_1$, every world that has at least one accessible world accesses a world in a different logic. Specifically, $a$, $b$, $c$, and $e$ are worlds with accessible worlds in different logics: $a$ accesses $f$, $b$ accesses $c$, $c$ accesses $d$, and $e$ accesses $f$. In contrast, frames $F_2$ and $F_3$ do not exhibit the out-of-bubble property. In $F_2$, $e$ accesses another world, but not one operating in a different logic. Similarly, in $F_3$, $c$ accesses a world, but only one of the same logical kind.

\begin{notation}
    We use $\top$ to refer to the formula $\circ p \lor \lnot \circ p$ for a given variable $p$, and we use $\bot$ to refer to the formula $\circ p \land \lnot \circ p$.
\end{notation}

\begin{theorem}\label{outofbubble}
    $F$ is out of the bubble if and only if $F \vDash \Diamond \top \so (\Box \lnot \circ \varphi \so \lnot \Box \varphi)$.
\end{theorem}

\begin{proof}
    Suppose $F \vDash \Diamond \top \so (\Box \lnot \circ x \so \lnot \Box x)$ and that $F$ is not out of the bubble. 
    As a result, there is a world $w$ in a model $M$ with frame $F$ such that $w$ accesses at least one world and every world $w'$ accessed by $w$ is such that $I(w) = I(w') = \{0,1,a,-a\}$. Then it follows that $w \vDash \Diamond \top$. Now, consider the valuation $v$ such that, for all $u$ in $F$, $v_u(x) = a$. Consequently, $v_w(\Box \lnot \circ x) = 1$ and $w \vDash \Box x$. This is absurd since $w \vDash \Box \lnot \circ x \so \lnot \Box x$.

    Now suppose $F$ is out of the bubble and $F \nvDash \Diamond \top \so (\Box \lnot \circ \varphi \so \lnot \Box \varphi)$. So there is a model $M$ with world $w$ with lattice $A = \{0,1,a,-a\}$ such that $w \vDash \Diamond \top$, $w \vDash \Box \lnot \circ \varphi$ and $w \vDash \Box \varphi$. Since $w \vDash \Diamond \top$, $w$ has access to at least one world. So, there is $w'$ in $F$ with a lattice different from the lattice in $w$ such that $w R w'$. As a result, $v_{w'}(\circ \varphi) = 0$ and, because it has a different lattice, $(v_{w'}(\varphi))^A = 0$. Thus, it follows that $v_w(\Box \varphi) = 0$, which contradicts $w \vDash \Box \varphi$.
\end{proof}

We have characterized with \cref{outofbubble} the property of accessing worlds operating in different lattices. However, our system includes three distinct Boolean algebras. Can we extend this characterization to account for accessibility between worlds of both kinds? To explore this, let us first define the property precisely:

\begin{definition}
    A frame $F$ is \textbf{super} out of the bubble when 
    \begin{enumerate}
        \item it is out of the bubble and,
        \item if $w$ in $F$ accesses a world $w'$ such that $I(w') \neq I(w)$, then $w$ accesses $w''$ such that  $I(w) \neq I(w') \neq I(w'') \neq I(w)$.
    \end{enumerate}
\end{definition}

This definition is similar to the out of bubble property. The difference, however, is that in this case every world that accesses at least one other world must access worlds of two distinct types. In our specific setting, in a super out of bubble frame, a world of lattice $A$ that has accessible world must also access at least one world from lattice $B$ and at least one world from lattice $C$.

\begin{theorem}\label{soobNotCharacterizable}
    No formula characterizes super out of the bubble.
\end{theorem}

\begin{proof}
    We will produce a frame $F$ that is super out of the bubble and a frame $F'$ that is not, and we will show that $F \vDash \varphi$ if and only if $F' \vDash \varphi$ for every formula $\varphi$. Assuming this, suppose that there is a formula $\theta$ that characterizes super out of the bubble. Because $F$ is super out of the bubble, $F \vDash \theta$; From the equivalence, this implies $F' \vDash \theta$; since $\theta$ characterizes super out of the bubble, we conclude that $F'$ is super out of the bubble, which is a contradiction.

    Consider the following frame $F$:
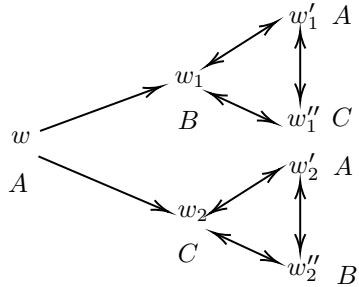
    \begin{center}
        \tikzset{every picture/.style={line width=0.75pt}} %set default line width to 0.75pt        
        \begin{tikzpicture}[x=0.75pt,y=0.75pt,yscale=-.8,xscale=.8]
        %uncomment if require: \path (0,300); %set diagram left start at 0, and has height of 300
        
        %Straight Lines [id:da05928991953527163] 
        \draw    (185.5,185) -- (265.66,219.21) ;
        \draw [shift={(267.5,220)}, rotate = 203.11] [color={rgb, 255:red, 0; green, 0; blue, 0 }  ][line width=0.75]    (10.93,-3.29) .. controls (6.95,-1.4) and (3.31,-0.3) .. (0,0) .. controls (3.31,0.3) and (6.95,1.4) .. (10.93,3.29)   ;
        %Straight Lines [id:da10004678245610554] 
        \draw    (186.5,169) -- (260.62,141.69) ;
        \draw [shift={(262.5,141)}, rotate = 159.78] [color={rgb, 255:red, 0; green, 0; blue, 0 }  ][line width=0.75]    (10.93,-3.29) .. controls (6.95,-1.4) and (3.31,-0.3) .. (0,0) .. controls (3.31,0.3) and (6.95,1.4) .. (10.93,3.29)   ;
        %Straight Lines [id:da7434380713720259] 
        \draw    (293.19,126.93) -- (333.81,101.07) ;
        \draw [shift={(335.5,100)}, rotate = 147.53] [color={rgb, 255:red, 0; green, 0; blue, 0 }  ][line width=0.75]    (10.93,-3.29) .. controls (6.95,-1.4) and (3.31,-0.3) .. (0,0) .. controls (3.31,0.3) and (6.95,1.4) .. (10.93,3.29)   ;
        \draw [shift={(291.5,128)}, rotate = 327.53] [color={rgb, 255:red, 0; green, 0; blue, 0 }  ][line width=0.75]    (10.93,-3.29) .. controls (6.95,-1.4) and (3.31,-0.3) .. (0,0) .. controls (3.31,0.3) and (6.95,1.4) .. (10.93,3.29)   ;
        %Straight Lines [id:da10677767127325755] 
        \draw    (295.19,219.93) -- (335.81,194.07) ;
        \draw [shift={(337.5,193)}, rotate = 147.53] [color={rgb, 255:red, 0; green, 0; blue, 0 }  ][line width=0.75]    (10.93,-3.29) .. controls (6.95,-1.4) and (3.31,-0.3) .. (0,0) .. controls (3.31,0.3) and (6.95,1.4) .. (10.93,3.29)   ;
        \draw [shift={(293.5,221)}, rotate = 327.53] [color={rgb, 255:red, 0; green, 0; blue, 0 }  ][line width=0.75]    (10.93,-3.29) .. controls (6.95,-1.4) and (3.31,-0.3) .. (0,0) .. controls (3.31,0.3) and (6.95,1.4) .. (10.93,3.29)   ;
        %Straight Lines [id:da7392138944965345] 
        \draw    (295.31,145.84) -- (332.69,163.16) ;
        \draw [shift={(334.5,164)}, rotate = 204.86] [color={rgb, 255:red, 0; green, 0; blue, 0 }  ][line width=0.75]    (10.93,-3.29) .. controls (6.95,-1.4) and (3.31,-0.3) .. (0,0) .. controls (3.31,0.3) and (6.95,1.4) .. (10.93,3.29)   ;
        \draw [shift={(293.5,145)}, rotate = 24.86] [color={rgb, 255:red, 0; green, 0; blue, 0 }  ][line width=0.75]    (10.93,-3.29) .. controls (6.95,-1.4) and (3.31,-0.3) .. (0,0) .. controls (3.31,0.3) and (6.95,1.4) .. (10.93,3.29)   ;
        %Straight Lines [id:da8351408220439647] 
        \draw    (297.31,233.84) -- (334.69,251.16) ;
        \draw [shift={(336.5,252)}, rotate = 204.86] [color={rgb, 255:red, 0; green, 0; blue, 0 }  ][line width=0.75]    (10.93,-3.29) .. controls (6.95,-1.4) and (3.31,-0.3) .. (0,0) .. controls (3.31,0.3) and (6.95,1.4) .. (10.93,3.29)   ;
        \draw [shift={(295.5,233)}, rotate = 24.86] [color={rgb, 255:red, 0; green, 0; blue, 0 }  ][line width=0.75]    (10.93,-3.29) .. controls (6.95,-1.4) and (3.31,-0.3) .. (0,0) .. controls (3.31,0.3) and (6.95,1.4) .. (10.93,3.29)   ;
        %Straight Lines [id:da12403182801219348] 
        \draw    (350.5,148) -- (350.5,108) ;
        \draw [shift={(350.5,106)}, rotate = 90] [color={rgb, 255:red, 0; green, 0; blue, 0 }  ][line width=0.75]    (10.93,-3.29) .. controls (6.95,-1.4) and (3.31,-0.3) .. (0,0) .. controls (3.31,0.3) and (6.95,1.4) .. (10.93,3.29)   ;
        \draw [shift={(350.5,150)}, rotate = 270] [color={rgb, 255:red, 0; green, 0; blue, 0 }  ][line width=0.75]    (10.93,-3.29) .. controls (6.95,-1.4) and (3.31,-0.3) .. (0,0) .. controls (3.31,0.3) and (6.95,1.4) .. (10.93,3.29)   ;
        %Straight Lines [id:da6892995026591608] 
        \draw    (350.5,243) -- (350.5,203) ;
        \draw [shift={(350.5,201)}, rotate = 90] [color={rgb, 255:red, 0; green, 0; blue, 0 }  ][line width=0.75]    (10.93,-3.29) .. controls (6.95,-1.4) and (3.31,-0.3) .. (0,0) .. controls (3.31,0.3) and (6.95,1.4) .. (10.93,3.29)   ;
        \draw [shift={(350.5,245)}, rotate = 270] [color={rgb, 255:red, 0; green, 0; blue, 0 }  ][line width=0.75]    (10.93,-3.29) .. controls (6.95,-1.4) and (3.31,-0.3) .. (0,0) .. controls (3.31,0.3) and (6.95,1.4) .. (10.93,3.29)   ;
        
        % Text Node
        \draw (166.5,167.4) node [anchor=north west][inner sep=0.75pt]    {$w$};
        % Text Node
        \draw (269.5,129.4) node [anchor=north west][inner sep=0.75pt]    {$w_{1}$};
        % Text Node
        \draw (271.5,214.4) node [anchor=north west][inner sep=0.75pt]    {$w_{2}$};
        % Text Node
        \draw (341.5,86.4) node [anchor=north west][inner sep=0.75pt]    {$w'_{1}$};
        % Text Node
        \draw (340.5,151.4) node [anchor=north west][inner sep=0.75pt]    {$w''_{1}$};
        % Text Node
        \draw (340.5,245.4) node [anchor=north west][inner sep=0.75pt]    {$w''_{2}$};
        % Text Node
        \draw (341.5,181.4) node [anchor=north west][inner sep=0.75pt]    {$w'_{2}$};
        % Text Node
        \draw (164.5,194.4) node [anchor=north west][inner sep=0.75pt]    {$A$};
        % Text Node
        \draw (368.5,88.4) node [anchor=north west][inner sep=0.75pt]    {$A$};
        % Text Node
        \draw (369.5,183.4) node [anchor=north west][inner sep=0.75pt]    {$A$};
        % Text Node
        \draw (271.5,154.4) node [anchor=north west][inner sep=0.75pt]    {$B$};
        % Text Node
        \draw (371.5,252.4) node [anchor=north west][inner sep=0.75pt]    {$B$};
        % Text Node
        \draw (271.5,238.4) node [anchor=north west][inner sep=0.75pt]    {$C$};
        % Text Node
        \draw (368.5,152.4) node [anchor=north west][inner sep=0.75pt]    {$C$};

        \end{tikzpicture}
    \end{center}
    As the reader can observe, each world has the lattice $A = \{1,0,a,-a\}$, $B = \{1,0,b,-b\}$ or $C = \{1,0,c,-c\}$ written next to them. These are the lattices of each world. Now, observe that this frame is super out of the bubble. Consider then the \textbf{similar} frame $F'$:
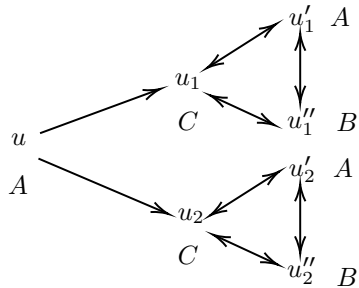
    \begin{center}
        \tikzset{every picture/.style={line width=0.75pt}} %set default line width to 0.75pt        
        \begin{tikzpicture}[x=0.75pt,y=0.75pt,yscale=-.8,xscale=.8]
        %uncomment if require: \path (0,300); %set diagram left start at 0, and has height of 300
        
        %Straight Lines [id:da05928991953527163] 
        \draw    (185.5,185) -- (265.66,219.21) ;
        \draw [shift={(267.5,220)}, rotate = 203.11] [color={rgb, 255:red, 0; green, 0; blue, 0 }  ][line width=0.75]    (10.93,-3.29) .. controls (6.95,-1.4) and (3.31,-0.3) .. (0,0) .. controls (3.31,0.3) and (6.95,1.4) .. (10.93,3.29)   ;
        %Straight Lines [id:da10004678245610554] 
        \draw    (186.5,169) -- (260.62,141.69) ;
        \draw [shift={(262.5,141)}, rotate = 159.78] [color={rgb, 255:red, 0; green, 0; blue, 0 }  ][line width=0.75]    (10.93,-3.29) .. controls (6.95,-1.4) and (3.31,-0.3) .. (0,0) .. controls (3.31,0.3) and (6.95,1.4) .. (10.93,3.29)   ;
        %Straight Lines [id:da7434380713720259] 
        \draw    (293.19,126.93) -- (333.81,101.07) ;
        \draw [shift={(335.5,100)}, rotate = 147.53] [color={rgb, 255:red, 0; green, 0; blue, 0 }  ][line width=0.75]    (10.93,-3.29) .. controls (6.95,-1.4) and (3.31,-0.3) .. (0,0) .. controls (3.31,0.3) and (6.95,1.4) .. (10.93,3.29)   ;
        \draw [shift={(291.5,128)}, rotate = 327.53] [color={rgb, 255:red, 0; green, 0; blue, 0 }  ][line width=0.75]    (10.93,-3.29) .. controls (6.95,-1.4) and (3.31,-0.3) .. (0,0) .. controls (3.31,0.3) and (6.95,1.4) .. (10.93,3.29)   ;
        %Straight Lines [id:da10677767127325755] 
        \draw    (295.19,219.93) -- (335.81,194.07) ;
        \draw [shift={(337.5,193)}, rotate = 147.53] [color={rgb, 255:red, 0; green, 0; blue, 0 }  ][line width=0.75]    (10.93,-3.29) .. controls (6.95,-1.4) and (3.31,-0.3) .. (0,0) .. controls (3.31,0.3) and (6.95,1.4) .. (10.93,3.29)   ;
        \draw [shift={(293.5,221)}, rotate = 327.53] [color={rgb, 255:red, 0; green, 0; blue, 0 }  ][line width=0.75]    (10.93,-3.29) .. controls (6.95,-1.4) and (3.31,-0.3) .. (0,0) .. controls (3.31,0.3) and (6.95,1.4) .. (10.93,3.29)   ;
        %Straight Lines [id:da7392138944965345] 
        \draw    (295.31,145.84) -- (332.69,163.16) ;
        \draw [shift={(334.5,164)}, rotate = 204.86] [color={rgb, 255:red, 0; green, 0; blue, 0 }  ][line width=0.75]    (10.93,-3.29) .. controls (6.95,-1.4) and (3.31,-0.3) .. (0,0) .. controls (3.31,0.3) and (6.95,1.4) .. (10.93,3.29)   ;
        \draw [shift={(293.5,145)}, rotate = 24.86] [color={rgb, 255:red, 0; green, 0; blue, 0 }  ][line width=0.75]    (10.93,-3.29) .. controls (6.95,-1.4) and (3.31,-0.3) .. (0,0) .. controls (3.31,0.3) and (6.95,1.4) .. (10.93,3.29)   ;
        %Straight Lines [id:da8351408220439647] 
        \draw    (297.31,233.84) -- (334.69,251.16) ;
        \draw [shift={(336.5,252)}, rotate = 204.86] [color={rgb, 255:red, 0; green, 0; blue, 0 }  ][line width=0.75]    (10.93,-3.29) .. controls (6.95,-1.4) and (3.31,-0.3) .. (0,0) .. controls (3.31,0.3) and (6.95,1.4) .. (10.93,3.29)   ;
        \draw [shift={(295.5,233)}, rotate = 24.86] [color={rgb, 255:red, 0; green, 0; blue, 0 }  ][line width=0.75]    (10.93,-3.29) .. controls (6.95,-1.4) and (3.31,-0.3) .. (0,0) .. controls (3.31,0.3) and (6.95,1.4) .. (10.93,3.29)   ;
        %Straight Lines [id:da12403182801219348] 
        \draw    (350.5,148) -- (350.5,108) ;
        \draw [shift={(350.5,106)}, rotate = 90] [color={rgb, 255:red, 0; green, 0; blue, 0 }  ][line width=0.75]    (10.93,-3.29) .. controls (6.95,-1.4) and (3.31,-0.3) .. (0,0) .. controls (3.31,0.3) and (6.95,1.4) .. (10.93,3.29)   ;
        \draw [shift={(350.5,150)}, rotate = 270] [color={rgb, 255:red, 0; green, 0; blue, 0 }  ][line width=0.75]    (10.93,-3.29) .. controls (6.95,-1.4) and (3.31,-0.3) .. (0,0) .. controls (3.31,0.3) and (6.95,1.4) .. (10.93,3.29)   ;
        %Straight Lines [id:da6892995026591608] 
        \draw    (350.5,243) -- (350.5,203) ;
        \draw [shift={(350.5,201)}, rotate = 90] [color={rgb, 255:red, 0; green, 0; blue, 0 }  ][line width=0.75]    (10.93,-3.29) .. controls (6.95,-1.4) and (3.31,-0.3) .. (0,0) .. controls (3.31,0.3) and (6.95,1.4) .. (10.93,3.29)   ;
        \draw [shift={(350.5,245)}, rotate = 270] [color={rgb, 255:red, 0; green, 0; blue, 0 }  ][line width=0.75]    (10.93,-3.29) .. controls (6.95,-1.4) and (3.31,-0.3) .. (0,0) .. controls (3.31,0.3) and (6.95,1.4) .. (10.93,3.29)   ;
        
        % Text Node
        \draw (166.5,167.4) node [anchor=north west][inner sep=0.75pt]    {$u$};
        % Text Node
        \draw (269.5,129.4) node [anchor=north west][inner sep=0.75pt]    {$u_{1}$};
        % Text Node
        \draw (271.5,214.4) node [anchor=north west][inner sep=0.75pt]    {$u_{2}$};
        % Text Node
        \draw (341.5,86.4) node [anchor=north west][inner sep=0.75pt]    {$u'_{1}$};
        % Text Node
        \draw (340.5,151.4) node [anchor=north west][inner sep=0.75pt]    {$u''_{1}$};
        % Text Node
        \draw (340.5,245.4) node [anchor=north west][inner sep=0.75pt]    {$u''_{2}$};
        % Text Node
        \draw (341.5,181.4) node [anchor=north west][inner sep=0.75pt]    {$u'_{2}$};
        % Text Node
        \draw (164.5,194.4) node [anchor=north west][inner sep=0.75pt]    {$A$};
        % Text Node
        \draw (367.5,88.4) node [anchor=north west][inner sep=0.75pt]    {$A$};
        % Text Node
        \draw (369.5,183.4) node [anchor=north west][inner sep=0.75pt]    {$A$};
        % Text Node
        \draw (371.5,154.4) node [anchor=north west][inner sep=0.75pt]    {$B$};
        % Text Node
        \draw (371.5,252.4) node [anchor=north west][inner sep=0.75pt]    {$B$};
        % Text Node
        \draw (271.5,238.4) node [anchor=north west][inner sep=0.75pt]    {$C$};
        % Text Node
        \draw (271.5,153.4) node [anchor=north west][inner sep=0.75pt]    {$C$};

\end{tikzpicture}
    \end{center}

    Observe that the frame $F'$ is not super out of the bubble since $u$ accesses only worlds with lattice $C$.

    Consider the sets of worlds $D_1 = \{w_1, w_1', w_1''\}$, $D_2 = \{w_2, w_2', w_2''\}$, $E_1 = \{u_1, u_1', u_1''\}$ and $E_2 = \{u_2, u_2', u_2''\}$. Each of these frames is exactly the same, i.e. a frame composed of three worlds, each with a different Boolean algebra ($A$, $B$, and $C$), and every world accesses each other. Therefore, we obtain
    $$D_1 \vDash \psi \miff D_2 \vDash \psi \miff E_1 \vDash \psi \miff E_2 \vDash \psi$$
    We prove $F \vDash \varphi$ if, and only if, $F' \vDash \varphi$. Suppose that 
    $$F \vDash \varphi \text{ and } F' \nvDash \varphi$$
    Consequently, there is a world $d$ in a model $M$ with frame $F'$ and valuation $v$ such that $d \nvDash \varphi$. Notably, $d$ cannot be any of the worlds in $E_1$ and $E_2$ since $F \vDash \varphi$ implies $E_1 \vDash \varphi$ and $E_2 \vDash \varphi$. Therefore, $d$ is the world $u$.

    Consider the transformation $eq$ from Boolean values and worlds to Boolean values:
    \begin{enumerate}
        \item $eq(1, x) = 1$
        \item $eq(0, x) = 0$
        \item Let $k \in \{a, b, c\}$, 
        $$eq(k, x) = \begin{cases}
            a, I(x) = A\\
            b, I(x) = B\\
            a, I(x) = C
        \end{cases}$$
        \item Let $k \in \{-a, -b, -c\}$, 
        $$eq(k, x) = \begin{cases}
            -a, I(x) = A\\
            -b, I(x) = B\\
            -a, I(x) = C
        \end{cases}$$
    \end{enumerate}

    Now consider the model $N$ with frame $F$ and valuation $h$ such that 
        $$h_{w_i^*}(p) = eq(v_{u_i^*}(p))$$
    being $p$ any variable, $i \in \{\emptyset, 1, 2\}$ and $*$ the symbol $'$ zero, one or two times.

    From this it follows that 
    \begin{enumerate}
        \item $h_{w_2}(\psi) = v_{u_2}(\psi)$ for all $\psi$.
        \item $h_{w_1}(\psi) = eq(v_{u_1}(\psi), w_1)$ for all $\psi$.
    \end{enumerate}

    From this we obtain the following. 
    \begin{enumerate}
        \item $(h_{w_2}(\psi))^A = (v_{u_2}(\psi))^A$ for all $\psi$.
        \item $(h_{w_1}(\psi))^A = (v_{u_1}(\psi))^A$ for all $\psi$.
    \end{enumerate}
    And thus we obtain that 
    $$h_w(\varphi) = h_u(\varphi)$$
    and this contradicts $F \vDash \varphi$ and $u \nvDash \varphi$.

    A similar procedure leads the supposition $F \nvDash \varphi$ and $F' \vDash \varphi$ to a contradiction. Hence, we have $F \vDash \varphi$ if and only if $F' \vDash \varphi$.
\end{proof}

The issue in the proof of Theorem~\ref{soobNotCharacterizable} is that worlds $u_1$ and $u_2$ are not connected. So, no difference can be extracted from $u_1$ and $u_2$ having different lattices from the point of view of $u$. We may, however, attempt to remedy this once we also have the Euclidean axiom. The axiom guarantees in the proof of Theorem~\ref{soobNotCharacterizable} world $u_1$ accesses $u_2$. However, this accessibility relation will not be sufficient to prevent $u_1$ and $u_2$ from having the same lattice. 

\begin{theorem}
    No formula characterizes super out of the bubble (SooB) and Euclidean.
\end{theorem}

\begin{proof}
    Suppose that there is $\varphi$ such that $F \vDash \varphi$ if and only if $F$ is SooB and Euclidean. Because there are SooB and Euclidean frames, $\{\varphi, 5^\circ\}$ is locally consistent. Then Theorem~\ref{valuation_canonical_model_2lattices} gives us a canonical model $M$ with respect to $\{\varphi, 5^\circ\}$ and a model $M'$ restricting $M$ to worlds with lattices $A$ and $B$. Consequently, $M' \vDash \varphi$ and $M \vDash 5^\circ$ while $M'$ is not SooB.\footnote{This same strategy could have been used to prove \cref{soobNotCharacterizable} to obtain a shorter proof. Nonetheless, we consider the more extensive proof found in \cref{soobNotCharacterizable} to be instructive, as it is independent of the canonical model construction.}
\end{proof}

Although our new framework allows for the derivation of new frame characterizations, these remain constrained by the expressive power of the underlying modal language. The primary difficulty is the limited expressive power of the specified necessity operator: it does not differentiate between the lattices of accessible worlds (e.g., from the perspective of a world with lattice $A$, worlds with lattices $B$ and $C$ appear indistinguishable). Moreover, the necessity operator alone cannot identify when two worlds share the same lattice (e.g., a world with lattice $A$ cannot determine that another world also has lattice $A$). In light of these limitations, we now briefly turn to further language extensions that enable the characterization of other, more expressive, properties.

\begin{definition}
    We say that a frame $F$ is \textbf{transitive through equality} when for every $w$, $u$ and $u'$ such that $I(w) = I(u) = I(u')$ and $w R u R u'$ implies $w R u'$.
\end{definition}

\begin{definition} The operator $\bBox$ is such that, for every formula $\varphi$ and world $w$ in every model $M$,
    $$v_w(\bBox \varphi) = \bigwedge \{v_{w'}(\varphi) \mid I(w') = I(w) \text{ and } w R w'\}$$
    In other words, $v_w(\bBox \varphi)$ is obtained by only considering the accessible worlds with the same lattice as $w$.
\end{definition}

\begin{theorem}
    $F$ is transitive through equality if and only if 
    $$ F \vDash \Box \varphi \so \bBox \bBox \varphi$$
\end{theorem}

\begin{proof}
    Assume that there exists a frame $F$ that is not transitive through equality and that $F \vDash \Box x \so \bBox \bBox x$. Consequently, there are $w R u R u'$ with $I(w) = I(u) = I(u')$ such that $w \not R u'$. We define the valuation $v$:
    $$v_e(p) = \begin{cases}
        0, \text{ if } e = u' \\
        1, \text{ otherwise}
    \end{cases}$$
    Since $I(u) = I(u')$ and $u R u'$, we obtain $v_u(\bBox x) = 0$. Now, because $I(w) = I(u)$ and $w R u$, we obtain $v_w(\bBox \bBox x) = 0$. Furthermore, $w \not R w'$ and all worlds other than $w'$ attribute the value $1$ to $x$, so we have $v_w(\Box x) = 1$. Then we obtain the absurd $w \nvDash \Box x \so \bBox \bBox x$.

    Suppose that $F$ is transitive through equality and there is $w$ in $M$ with frame $F$ such that $w \vDash \Box \varphi$ and $w \nvDash \bBox \bBox \varphi$. Assume without loss of generality that $I(w) = A$. Since $w \nvDash \bBox \bBox \varphi$, there is $u$ such that $I(w) = I(u)$ and $v_u(\bBox \varphi)$ is $0$ or $-a$. Consequently, there is $u'$ such that $u R u'$, $I(u) = I(u') = A$ and $v_{u'}(\varphi)$ is $0$ or $-a$.
    Since $F$ is transitive through equality, we obtain $w R w'$ and $v_w(\Box \varphi) = 0$. This is absurd since $w \vDash \Box \varphi$.
\end{proof}

\begin{definition}
    We say that a frame $F$ is \textbf{transitive through difference} when for every $w$, $u$ and $u'$ such that
    $I(w) \neq I(u) = I(u')$, $w R u R u'$ implies $w R u'$.
\end{definition}

\begin{definition} The operator $\mBox$ is such that, for every formula $\varphi$ and world $w$ in every model $M$,
    $$v_w( \mBox\varphi) = \bigwedge \{(v_{w'}(\varphi))^{I(w)} \mid I(w') \neq I(w) \text{ and } w R w'\}$$
    In other words, $v_w( \mBox\varphi)$ is obtained by considering only the accessible worlds that do not have the same lattice as $w$.
\end{definition} 

\begin{theorem}
    $F$ is transitive through difference if and only if 
    $$F \vDash \Box \varphi \so \mBox \bBox (\circ \varphi \land \varphi)$$
\end{theorem}

\begin{proof}
    Suppose $F \vDash \Box x \so \mBox \bBox (\circ x \land x)$ and that $F$ is not transitive through difference. Then there are $w R u R u'$ with $I(w) \neq I(u)$ and $I(w) = I(u')$ such that $w \not R u'$. We define the valuation $v$:
    $$v_e(p) = \begin{cases}
        0, \text{ if } e = u' \\
        1, \text{ otherwise}
    \end{cases}$$
    It follows that $v_{u'}(\circ x \land x) = 0$. Since $I(u) = I(u')$ and $u R u'$, we obtain $v_u(\bBox (\circ x \land x)) = 0$ and, from $I(w) \neq I(u)$ and $w R u$, that $v_w(\mBox \bBox (\circ x \land x)) = 0$. Furthermore, $w \not R w'$ and all worlds other than $w'$ attribute the value $1$ to $x$, so we have $v_w(\Box x) = 1$. Thence $w \nvDash \Box x \so \mBox \bBox (\circ x \land x)$, which is absurd.

    Suppose that $F$ is transitive through difference and that there is $w$ in $M$ with frame $F$ such that $w \vDash \Box \varphi$ and $w \nvDash \mBox \bBox (\circ \varphi \land \varphi)$. Notably, $\circ \varphi \land \varphi$ can only result in values $0$ or $1$ and, since $w \nvDash \mBox \bBox (\circ \varphi \land \varphi)$, there is $u$ such that $I(w) \neq I(u)$ and $v_u(\bBox (\circ \varphi \land \varphi)) =0$. Consequently, there is $u'$ such that $u R u'$, $I(u) = I(u')$ and $v_{u'}(\circ \varphi \land \varphi) = 0$. From $v_{w'}(\circ \varphi \land \varphi) = 0$, we obtain $v_{w'}(\varphi) \neq 1$. 
    Because it is transitive through difference, $w R w'$. For $I(w) \neq I(u')$ and $w R w'$, we obtain $v_w(\Box \varphi) = 0$, which contradicts $w \vDash \Box \varphi$.
\end{proof}

\section{Conclusions and further work}

In this work, we explored the consistency operator within a novel framework, focusing on its application to non-normal modal consequence relations over Kripke frames. Departing from traditional approaches, we utilize modal logic structures based on varied (yet classical) base lattices to implement this idea effectively.

A central aspect of our approach is the analysis of modal logic over a four-valued Boolean algebra. Specifically, we study a system in which a four-valued logic is embedded within an eight-valued Boolean algebra. Building on the framework introduced in \cite{freire_modality_2024}, we define a many-logic modal logic by taking the $8$-valued Boolean algebra $B_8$ as the underlying lattice $L$, and considering $LAT$ as the collection of three different four-valued subalgebras of $B_8$. We provide a sound and complete axiomatization of this system.

The final section investigates the relationship between our system and the classical global consequence. We analyze various frame properties and show that traditional characterizations do not always hold. For example, the Euclidean property cannot be captured by the usual formula (Theorem \ref{Euc_no}). However, we demonstrate that a formula involving the consistency operator can still characterize Euclidean frames (Theorem \ref{Euc_yes}). However, some properties resist full characterization even with the consistency operator, indicating the need for an extended language incorporating additional modal operators. Two such cases are discussed at the end of Section \ref{sec_prop}.

In this context, we propose studying conditional versions of standard frame properties, such as:

\begin{enumerate} 
\item \textbf{$L_0$-reflexivity:} For all $w \in W$, if $L_w = L_0$ then $w R w$. 
\item \textbf{Symmetry on agree-worlds:} For all $u, w \in W$ such that $L_u = L_w$, if $u R w$ then $w R u$. 
\end{enumerate}

The exploration of this kind of property allows us to address conflicting perspectives regarding the meaning of logical connectives. As Quine argued in \cite{quine_philosophy_1986}, debates about logical connectives are meaningless since the logical vocabulary itself constitutes the very framework in which such discussions take place. However, following the arguments in \cite{freire_modality_2024}, we consider the possibility of a partial transfer of semantic information through common lattices. This transfer opens space for meaningful discussions about the interpretation of logical connectives. In the example examined in this paper, we do not encounter an explicit disagreement about logical connectives, but rather an implicit one: classical logicians disagree about the uncertain validity of evaluations made by other classical logicians. Deeper disagreements are also possible, as illustrated in \cite{freire_modality_2024}, where some worlds operate under binary classical logic while others employ the three-valued logic of paradox. A broader exploration of this topic may provide foundations for the literature on counterpossibility, offering tools to evaluate the ``possibility of having a different logic.''

We also suggest extending the study of frame properties characterized by modal formulas to other many-logic modal structures. A particularly interesting case is that of finite linear lattices, where one can naturally define the notion of increasing classicality. For example, given sublattices of the linear lattice $[0,1]$ containing $0$ and $1$, we may say one sublattice is more classical than another if it is contained within it. In this framework, we can examine frames whose accessibility relations reflect increasing or decreasing classicality ---i.e., where $w R w'$ implies $L_w \subseteq L_{w'}$, or study transitivity restricted to paths that access increasingly classical worlds (i.e., $w R w' R w''$ and $L_w \subseteq L_{w'} \subseteq L_{w''}$ implies $w R w''$).

Finally, we argue that this framework offers a more natural semantic foundation for non-normal modal logics than traditional neighborhood semantics. As discussed in \cite{freire_modality_2024}, many non-normal phenomena emerge organically in the many-logic setting. For instance, in such a framework, the validity of a formula under necessitation does not necessarily imply its truth across all accessible worlds.

\medskip

\noindent {{\bf Acknowledgments.} This work is supported by FCT – Funda\c c\~ao para a Ci\^encia e a Tecnologia through
project UIDB/04106/2025 at CIDMA and and by the Project Bansky with reference compete2020-feder-00892000.}

\printbibliography

\end{document}